\documentclass[a4paper]{article}

\usepackage{times}

\usepackage{amssymb}
\usepackage{amsmath}  
\usepackage{amsthm} 
\usepackage{enumerate}  
\usepackage{latexsym}
\usepackage{a4wide}

\usepackage{tikz-cd}
\usetikzlibrary{shapes,arrows,decorations.markings,positioning}

\newtheorem{theorem}{THEOREM}[section]
\newtheorem{lemma}[theorem]{LEMMA}
\newtheorem{corollary}[theorem]{COROLLARY}

\newtheorem{proposition}[theorem]{PROPOSITION}

\newtheorem{problem}[theorem]{PROBLEM}
\newtheorem{example}[theorem]{EXAMPLE}
\newtheorem{defin}[theorem]{DEFINITION}

\newenvironment{definition}{\begin{defin}\rm}{\end{defin}}

\makeatletter
\newcommand*\rel@kern[1]{\kern#1\dimexpr\macc@kerna}
\newcommand*\widebar[1]{%
  \begingroup
  \def\mathaccent##1##2{%
    \rel@kern{0.8}%
    \overline{\rel@kern{-0.8}\macc@nucleus\rel@kern{0.2}}%
    \rel@kern{-0.2}%
  }%
  \macc@depth\@ne
  \let\math@bgroup\@empty \let\math@egroup\macc@set@skewchar
  \mathsurround\z@ \frozen@everymath{\mathgroup\macc@group\relax}%
  \macc@set@skewchar\relax
  \let\mathaccentV\macc@nested@a
  \macc@nested@a\relax111{#1}%
  \endgroup
}
\makeatother

\def\b #1{\bar{#1}}
\def\o #1{\widebar{#1}} 
\def\wh #1{\widehat{#1}}

\def\Claim#1{\par\medskip\noindent{\bf Claim #1. }}
\def\pfclaim{\par\noindent{\bf Proof of claim. }}

\def\dom{\mathop{\rm dom}}

\def\Vec#1#2{#1_0,\ldots{},#1_{#2}}

\def\c #1{{\cal #1}}
\def \gc#1{\mathfrak{#1}}

\def\str{structure}
\def\uf{ultrafilter}

\def\upwr{ultrapower}
\def\ba{boolean algebra}
\def\fo{first-order}

\def\:={\buildrel \rm def \over =}

\def\bo{\Box}
\def\di{\lozenge}
\def\bbo{\blacksquare}
\def\bdi{\blacklozenge}

\def\To{\Rightarrow}
\def\inn{\mathrel{\underline{\in}}}

\def\M{\mathsf{M}}
\def\Q{\mathsf{Q}}
\def\ci{{^\circ}}
\def\u #1{\underline{#1}}
\def\pdt{pairwise distinct}

\def\nice{{resolved}}
\def\ach{achronal}

\title{Canonicity in power and modal logics of finite \ach\ width}
\author{Robert Goldblatt\thanks{School of Mathematics and Statistics,
Victoria University, Wellington, New Zealand. 
sms.vuw.ac.nz/$\sim$rob}\; 
and Ian Hodkinson\thanks{Department of Computing, Imperial College London, UK\null.
www.doc.ic.ac.uk/$\sim$imh/}}

\begin{document}
\maketitle

\begin{abstract}

We develop a method for showing that various modal logics that are valid in their countably generated canonical Kripke frames  must also be valid in  their uncountably generated ones. This is applied to many systems,  including the  logics of finite width, and a broader class of multimodal logics of `finite achronal width' that are introduced here.


\end{abstract}

\paragraph{Keywords.}
Canonical modal logic; finite width modal logic.

\noindent{\bf MSC2020 classification.}
Primary: 03B45. Secondary: 06E25.

\section{Introduction}

There is a fundamental question about the model theory of propositional modal logics that has remained open since the early days of the subject. It was first raised in \cite{Fine75} and asks: if a  logic  is valid in its countably generated canonical Kripke frames,  must it also be valid in  its uncountably generated ones?

To elaborate this, note that the points of a canonical frame are maximally consistent sets of formulas. If the ambient language has $\kappa$ atoms/variables, there will be at least $2^\kappa$ such points.
This ``$\kappa$-generated'' frame carries a particular model that verifies all (and only) the theorems of the logic, but in some cases it may also carry models that falsify some theorems. If a logic is \emph{valid} in the frame, i.e.\ verified by all models on the frame, we say it is \emph{canonical in power $\kappa$}, or $\kappa$-\emph{canonical}. A logic canonical in all powers will be called \emph{totally canonical}.

 It is known that a $\kappa$-canonical logic is $\lambda$-canonical for all $\lambda<\kappa$, and that there exist logics that are canonical in all finite powers but not $\omega$-canonical.
Our fundamental question is: if a logic is $\omega$-canonical, must it be canonical in all uncountable powers, and hence totally canonical?

To put it another way: because the collection of logics forms a set, rather than a proper class, it can be seen that there does exist some infinite cardinal $\kappa$ such that every $\kappa$-canonical logic is totally canonical. Our question asks whether the least such $\kappa$ is $\omega$, or something larger.

There have been some positive answers given for limited classes of logics. Fine \cite{Fine75} made the important discovery that a logic must be totally canonical if it is characterised by (i.e.\ sound and complete for validity in)  a class of frames that is first-order definable. He applied this in \cite{fine:logi85} to \emph{subframe logics}, those that are characterised by a class that is closed under subframes.
It was shown that a subframe logic whose validating frames are transitive is totally canonical iff the class of these frames is first-order definable. Moreover that holds iff  the logic is \emph{compact} (alias \emph{strongly complete}), which means that any consistent set of formulas is satisfiable in a model on a frame that validates the logic. Now any $\omega$-canonical logic is compact, so it follows that any  $\omega$-canonical transitive subframe logic is totally canonical.

Zakharyaschev \cite{zakh:cano96} generalised these results, showing that they hold for any logic characterised by a class of transitive frames that is closed under \emph{cofinal} subframes, and that there are continuum many ($2^{\aleph_0}$) such logics that are not subframe logics.
Wolter \cite{wolt:stru97} removed the transitivity restriction in Fine's analysis.
He also studied linear temporal logics in \cite{wolt:prop96}, showing that they are totally canonical iff characterised by a first-order definable class, but that this condition is now stronger than compactness. The linear temporal logic of the real numbers is compact but not totally canonical. Nonetheless we show here that every  $\omega$-canonical linear temporal logic  is totally canonical.

The present paper develops a technique for giving a positive answer to the fundamental question for various logics, by a kind of L\"owenheim--Skolem argument that reduces a `big' failure of canonicity to a countably generated one. We combine  Kripke modelling with algebraic semantics, under which logics $L$ correspond to varieties (equationally definable classes) $V_L$  of algebras. Corresponding to the canonical frame construction there is an operation assigning to any modal algebra $\c A$ a frame $\c A_+$ whose points are the ultrafilters of $\c A$. The algebra $\c A^\sigma$ of all subsets of $\c A_+$ is the \emph{canonical extension} of $\c A$. When $\c A$ is the Lindenbaum algebra of $L$, relative to a language with $\kappa$ atoms, then $\c A_+$ is isomorphic to the $\kappa$-generated canonical frame of $L$, and $L$ is $\kappa$-canonical iff $\c A^\sigma\in V_L$.
 Moreover the Lindenbaum algebra is free in $V_L$ on $\kappa$ generators, and from this one can show that $L$ is $\kappa$-canonical iff $\c A^\sigma\in V_L$ for every $\c A\in V_L$ that is generated by a set of cardinality at most $\kappa$. We find it convenient to work with this algebraic formulation of canonicity \cite[\S 3.5]{gold:vari89}.

Let us say that a logic or variety is \emph{\nice} if it is either totally canonical, or not 
$\omega$-canonical (in which case it is not canonical in any infinite power). The fundamental question asks whether every logic is \nice.
We illustrate our strategy for tackling this by applying it to logics that contain the axiom $\di\di p\to\bo\di p$, which we call 5$_2$ (it is a weakening of the 5-axiom from the system S5). To explain the semantic meaning of 5$_2$ we take a temporal view of the binary relation $R$ of a frame and think of a set of the form $R(x)=\{z:Rxz\}$ as the \emph{future}  of point $x$. Then a frame validates 5$_2$ iff for each $x$, all points in $R(x)$ have the same future, i.e.\ $R(y)=R(y')$ for all  $y,y'\in R(x)$.
 
We show that any logic containing $5_2$ is \nice, and that there is a continuum of such logics. These results are then substantially generalised by introducing a new family 
$\{U_n:n<\omega\}$ of axioms that generalise the `finite width' axioms 
$\{I_n:n<\omega\}$ from \cite{+Fine74}. The relational condition for validity of $I_n$ is that no future set $R(x)$ contains an antichain with more than $n$ points. Here an antichain is a set $S$ such that $Ryz$ fails for all distinct $y,z\in S$, i.e.\ no member of $S$ is in the future of any other member. A finite width logic is one that includes $K4I_n$, the smallest normal logic to contain the transitivity axiom 4 and the axiom $I_n$, for some $n$.

$U_n$ is the formula
$$
\bigvee_{i\leq n}\bo\big(\bo q_i\to\bigvee\nolimits_{j\leq n,\,j\neq i}\bo q_j\big),
$$
which is derivable from 4 and $I_n$.
Its relational meaning is obtained by replacing antichains by the notion of an \emph{achronal} set $S$, one for which  $R(y)\not\subseteq R(z)$ for all distinct $y,z\in S$, i.e.\ no two points of $S$ have $\subseteq$-comparable futures. The condition for validity of $U_n$ is that no set  $R(x)$ contains an achronal set with more than $n$ points.

This discussion has assumed we are dealing with a single modality $\di$ and its dual $\bo$. In fact we allow ourselves a multimodal language with a set $\M$ of diamond modalities $\di,\bdi,\dots$ having duals $\bo,\bbo,\dots$, and take $U_n^\M$ to be the set of formulas
$$
\bigvee_{i\leq n}\bo\big(\bbo q_i\to\bigvee\nolimits_{j\leq n,\,j\neq i}\bbo q_j\big)
$$
for all $\di,\bdi\in\M$. We prove that any logic in this language that includes $U_n^\M$ is \nice.

In  the monomodal case $\M=\{\di\}$, this result includes all the finite width logics, but it covers much more. We show that there is a continuum of logics that include  $K4U_n$ and are included in $K4I_n$, and a continuum that include 
   $KU_{n+1}$ and are included in $KU_n$, for each $n$ (as well as other results in this vein).
 
 Fine showed in \cite{+Fine74} that any logic containing one of the finite width logics $K4I_n$ is \emph{Kripke complete}: it is complete for validity in its Kripke frames. This situation does not extend to the $KU_n$'s, or even the $K4U_n$'s, as we show by constructing an extension of $K4U_2$ that is  Kripke incomplete.

\subsection{Layout of paper}

\S\ref{sec:defs} recalls some relevant definitions
and sets out canonicity in power.
\S\ref{sec:setup} develops a strategy for proving
that certain modal logics (or varieties of BAOs) are \nice.
\S\ref{sec: K5}~applies this strategy to logics extending $K5_2$,
and shows that such extensions are numerous.

We take a break from canonicity
in sections~\ref{sec:KUn} and~\ref{sec:kr comp},
which can be read largely independently of the rest of the paper.
Motivated by the work on $K5_2$,
in \S\ref{sec:KUn} we introduce the modal logics $KU_n$,
for finite $n\geq1$, and compare them to some other logics.
Extensions of the $KU_n$ are called \emph{logics of finite \ach\ width,}
and the remainder of the paper studies them.
In \S\ref{sec:kr comp} we show that some of them are Kripke incomplete.
We return to canonicity in \S\ref{sec:main canon},
where we apply the strategy to them, showing that they are all \nice.

\section{Basic definitions}\label{sec:defs}

We start by recalling some standard modal and algebraic notions,
and setting up some definitions and notation for later use.

\subsection{General machinery used in the paper}\label{ss:strs}

We work in ZFC set theory and use standard set-theoretic material,
including maps or functions, ordinals, and cardinals.
A cardinal is an (initial) ordinal.
We write the first infinite cardinal as both $\omega$ and $\aleph_0$;
\emph{countable} will mean of cardinality at most this cardinal.
The cardinality of a set $X$ is denoted by $|X|$,
and the power set of $X$ by $\wp(X)$.
We write $\dom(f)$ for the domain of a map $f$.

\begin{definition}\label{def:bin rels}
Let $R$ be a binary relation on a set $X$ (ie.\ $R\subseteq X\times X$), and let $Y\subseteq X$.
\begin{enumerate}

\item We may write any of $xRy$, $Rxy$, $R(x,y)$ to denote that $(x,y)\in R$.
Later, we will use the same convention for \fo\ formulas.

\item For $x\in X$, we write $R(x)=\{y\in X:Rxy\}$.

Thinking of temporal logic,
we may speak of $R(x)$ as the \emph{($R$-)future of $x$.}

\item We write
$R[Y]=\{R(y):y\in Y\}$.

\item $Y$ is said to be \emph{$R$-closed (in $X\!$)} if $\bigcup R[Y]\subseteq Y$.
That is, if  $y\in Y$, $x\in X$, and $Ryx$, then $x\in Y$.

\item $Y$ is said to be an \emph{$R$-antichain}
if $\neg Ryz$ for all \emph{distinct} $y,z\in Y$.

\item 
The \emph{$R$-width} of $Y$ 
is the least cardinal $\kappa$
such that every $R$-antichain $Z\subseteq Y$ has cardinality
$\leq\kappa$.  
\end{enumerate}
\end{definition}

\subsection{Multimodal logic}\label{ss:mml}

We fix a nonempty 
set $\M$ of unary diamond-modalities,
and a set $\Q=\{q_i:i<\omega\}$ of propositional atoms.
 \emph{Modal $\M$-formulas,} or simply \emph{modal formulas,} 
 are then defined as usual.
Each atom in $\Q$ is a formula, as is $\top$,
and if $\varphi,\psi$ are formulas,
so are $\neg\varphi$, $\varphi\vee\psi$,
and $\di\varphi$, for each $\di\in\M$.

We adopt standard conventions.
We let $p,q,r,\ldots$ stand for arbitrary distinct atoms in~$\Q$.
We let
$\bot$ abbreviate $\neg\top$;
$\varphi\wedge\psi$ abbreviate $\neg(\neg\varphi\vee\neg\psi)$;
$\varphi\to\psi$ abbreviate $\neg\varphi\vee\psi$;
 $\varphi\leftrightarrow\psi$ abbreviate $(\varphi\to\psi)\wedge(\psi\to\varphi)$;
 and $\bo\varphi$ abbreviate $\neg\di\neg\varphi$, for each $\di\in\M$.
We adopt the usual binding conventions,
whereby $\neg,\di,\wedge,\vee,\to,\leftrightarrow$ (for each $\di\in\M$) are in decreasing order of tightness.

A \emph{normal modal $\M$-logic}
is a set of $\M$-formulas
containing all propositional tautologies,
the axioms $\bo(p\to q)\to(\bo p\to\bo q)$
and $\di p\leftrightarrow\neg\bo\neg p$
for each $\di\in\M$,\footnote{The latter axiom is needed because we take diamonds as
primitive and boxes as abbreviations. 
The same is done in \cite{BRV:ml}, for similar reasons; see \cite[p.34]{BRV:ml} for a discussion.}
 and closed under the inference rules of
modus ponens (from $\varphi$ and $\varphi\to\psi$, infer $\psi$),
generalisation (from $\varphi$, infer $\bo\varphi$),
and substitution (from $\varphi$, infer any
substitution instance of it --- any formula obtained
from $\varphi$ by replacing its atoms by arbitrary formulas).

\subsection{Structures}

We assume some familiarity with basic \fo\ model theory.
For unfamiliar terms, see, e.g., \cite{ChK90,Hodg93}.
A \emph{signature} is a set of non-logical symbols
(function symbols, relation symbols, and constants)
 together with their types and arities.
Others call it a similarity type, alphabet, or vocabulary.
 Equality is regarded as a logical symbol and is always available.
We will be using various model-theoretic \str s,
some of which will be two-sorted.
We take it as read that standard model-theoretic results apply
to two-sorted \str s.
We identify (notationally) a \str~$\gc M$ with its domain $\dom(\gc M)$.
We write $S^\gc M$ for the interpretation of a symbol
or term $S$ in $\gc M$.
We write $\b a\in\gc M$ to indicate that $\b a$ is a tuple of elements of the domain of $\gc M$.
For a map $f$ defined on $\gc M$,
we write $f(\b a)$ for the tuple obtained from $\b a$ by applying $f$ to its elements in order.

\subsection{Kripke frames, models, semantics}\label{ss:dual str}

Kripke frames are formulated as \str s, as follows.
We extend this slightly to cover Kripke models too.

\begin{definition}\label{def:frame}
 Introduce a binary relation symbol $R_\di$
for each $\di\in\M$.
A \emph{(Kripke) frame} is a \str\ for this signature.
A \emph{subframe} of a frame $\c F$ is simply a sub\str\ of~$\c F$ in the usual 
model-theoretic sense.
An \emph{inner subframe} of $\c F$ is a subframe
whose domain is $R_\di^\c F$-closed in $\c F$ 
(definition~\ref{def:bin rels}) for every $\di\in\M$.

A \emph{valuation into a frame $\c F$}
is a map $g:\Q\to\wp(\c F)$.
(Recall that notationally we identify $\c F$ with its domain.)

A \emph{(Kripke) model} is a pair $\c M=(\c F,g)$, where $\c F$ is a frame
and $g$ a valuation into it.
The \emph{frame of $\c M$} is $\c F$.
We may regard $\c M$ as a \str\ for the signature of frames 
by identifying it with $\c F$ (for this purpose).
So, for example, we write $t,u\in\c M$ to mean that $t,u\in\c F$,
and $\c M\models R_\di tu$ to mean that $\c F\models R_\di tu$, for $\di\in\M$.

A~\emph{submodel} of $\c M$
is a Kripke model of the form $(\c F',g')$,
where $\c F'$ is a subframe of $\c F$ and
$g'$ is the valuation into it given by
$g'(q)=g(q)\cap\c F'$, for each $q\in\Q$.
A submodel is fully determined by its domain.

\end{definition}

We now recall the standard modal definitions of truth, validity, etc.
\begin{definition}
Let $\c M=(\c F,g)$ be a Kripke model.
We define $\c M,w\models\varphi$
(in words, `$\varphi$ is true in $\c M$ at $w$'), 
for each $w\in\c F$ and modal formula $\varphi$, by induction as usual:
$\c M,w\models\top$;
$\c M,w\models q$ iff $w\in g(q)$, for $q\in\Q$;
$\c M,w\models\neg\varphi$ iff $\c M,w\not\models\varphi$;
$\c M,w\models\varphi\vee\psi$ iff
$\c M,w\models\varphi$ or $\c M,w\models\psi$ or both;
and if $\di\in\M$, then
$\c M,w\models\di\varphi$ iff $\c M,u\models\varphi$ for some
$u\in\c M$ with $\c M\models R_\di wu$.

We say that a modal formula
$\varphi$ is
\emph{satisfied in a Kripke model $\c M$}
if $\c M,w\models\varphi$ for some $w\in\c M$.
We will say that $\c M$ \emph{verifies} $\varphi$
 (written $\c M\models\varphi$),
if $\c M,w\models\varphi$ for every $w\in\c M$,
and that $\c M$ \emph{strongly verifies} $\varphi$
if it verifies every substitution instance of $\varphi$.

Let $\c F$ be a Kripke frame, and $\c K$ a class of frames.
We say that $\varphi$ is
\emph{satisfiable in $\c F$,}
if  it is satisfied in some Kripke model with frame $\c F$;
\emph{valid at a point $w$ in $\c F$,}
if $\c M,w\models\varphi$ for every Kripke model $\c M$ with frame $\c F$;
 \emph{valid in $\c F$}
(written $\c F\models\varphi$),
if every model with frame $\c F$ verifies $\varphi$;
and \emph{valid in $\c K$}
(written $\c K\models\varphi$), if $\varphi$ is valid in every frame in $\c K$.
In the latter cases, we also say that \emph{$\c F$ or $\c K$ validates $\varphi$.}

We say that a set $S$ of modal formulas is \emph{valid in $\c F$,}
and write $\c F\models S$,
if each formula in $S$ is valid in $\c F$.
The other kinds of validity are extended to sets of formulas similarly.

The set of all formulas valid in $\c F$ or $\c K$
is a normal modal $\M$-logic, called the \emph{logic of
(or the logic determined by)~$\c F$ or $\c K$.}

\end{definition}

\subsection{Point-generated inner subframes}

\begin{definition}\label{def:etc}
 For an ordinal $\alpha$, we  put $^\alpha \M=\{s\mid s:\alpha\to \M\}$, the
set of functions from $\alpha$ into $\M$,
and $^{<\alpha}\M=\bigcup_{\beta<\alpha}{}^\beta \M$.
One can think of the elements of $^{<\omega}\M$
as the finite sequences of diamonds.

 For an ordinal $\alpha$, a map $s\in{}^{\alpha}\M$, and $\di\in\M$,
we write $s\hat\ \di\in{}^{\alpha+1}\M$ for the map
whose restriction to $\alpha$ is $s$ and whose value on $\alpha$ is $\di$.
Thinking of $s$ as a sequence, $s\hat\ \di$ is the same sequence with $\di$ appended.

For each $s\in{}^{<\omega}\M$,
we define the \fo\ formula 
$R_s xy$ of the signature of frames, by induction on the ordinal $\dom(s)$:
\begin{itemize}
\item $R_\emptyset xy$ is $x=y$, where $\emptyset$ is the empty map in $^0\M$,

\item $R_{s\hat\ \di}xy$ is $\exists z(R_s xz\wedge R_\di zy)$.
\end{itemize}

  Finally, let  $\c F$ be a frame and $w_0\in\c F$.
The \emph{inner subframe of $\c F$ generated by $w_0$}
is the subframe $\c F(w_0)$
of $\c F$ with domain 
\[
\{w\in\c F:\c F\models R_s w_0w\mbox{ for some }s\in{}^{<\omega}\M
\}.
\]
\end{definition}
Taking $s=\emptyset$ here, we see that $\c F(w_0)$ contains $w_0$.
It is the smallest inner subframe of $\c F$ to do so.

\subsection{$\M$-BAOs and translations}\label{ss:algs}

It is now well known that modal logic can be `algebraised' using 
\emph{\ba s with operators (BAOs).}
We will just give some definitions;
for more information, see, e.g., \cite[chapter 5]{BRV:ml}.
We will pass fairly freely between modal logic and algebra.

An \emph{$\M$-BAO}
is an algebra $\c A$
with signature $\{+,-,0,1\}\cup\M$, where 
the $\{+,-,0,1\}$-reduct of $\c A$ is a \ba\ and
each $\di\in\M$ is
a unary function symbol interpreted in $\c A$ as
a normal additive operator 
(that is, $\c A\models\di0=0$ and $\c A\models\di(a+b)=\di a+\di b$ for each $a,b\in\c A$).
This is as per \cite[definition 5.19]{BRV:ml}.
The double use of $\di$ for a modal operator and a function symbol
should not cause problems in practice.

Standardly, each modal $\M$-formula
$\varphi$ can be translated to a term $\tau_\varphi$ of this signature
(an `$\M$-BAO term').
Introduce pairwise distinct variables $v_i$ ($i<\omega$).
We let $\tau_\top=1$,
$\tau_{q_i}=v_i$ for $i<\omega$, 
$\tau_{\neg\varphi}=-\tau_\varphi$,
$\tau_{\varphi\vee\psi}=\tau_\varphi+\tau_\psi$,
and $\tau_{\di\varphi}=\di\tau_\varphi$, for $\di\in\M$.
We can now say that a class 
$\c C$ of $\M$-BAOs \emph{validates} a modal formula $\varphi$
if $\c C\models\forall\b v(\tau_\varphi=1)$,
where $\b v$ is the tuple of variables occurring in $\tau_\varphi$.
As usual, $\c C$ validates a set of  modal formulas
if it validates every formula in the set.

A \emph{variety of $\M$-BAOs} is a
 class $V$ of $\M$-BAOs defined by a set of equations.
The set
$L_V$ of modal $\M$-formulas validated by $V$
is then a normal modal $\M$-logic as defined in \S\ref{ss:mml}.
If $L$ is a normal modal $\M$-logic,
$V_L$ denotes the variety of all $\M$-BAOs defined
by the set $\{\tau_\varphi=1:\varphi\in L\}$ of equations.
We have $V_{L_V}=V$ and $L_{V_L}=L$.
See, e.g., \cite[\S5.2]{BRV:ml} for more.

Conversely, each $\M$-BAO term $\tau$
written with the variables $v_i$
can be translated to a modal $\M$-formula $\varphi_\tau$.
We let $\varphi_1=\top$,
$\varphi_0=\bot$,
$\varphi_{v_i}=q_i$,
$\varphi_{-\tau}=\neg\varphi_\tau$,
$\varphi_{\tau+\upsilon}=\varphi_\tau\vee\varphi_{\upsilon}$,
and $\varphi_{\di\tau}=\di\varphi_\tau$ for $\di\in\M$.
We extend this translation to equations:
given $\M$-BAO terms $\tau,\upsilon$,
we let  $\varphi_{\tau=\upsilon}$ be the
modal formula
$\varphi_\tau\leftrightarrow\varphi_\upsilon$.

\subsection{Canonical extensions}

For an $\M$-BAO $\c A$, we can form, in the usual way,
\begin{itemize}
\item  its \emph{canonical frame} $\c A_+$,
whose domain comprises the \uf s of $\c A$,
and which is regarded as a frame in the  sense of 
definition~\ref{def:frame},
with interpretations given as usual
by $\c A_+\models R_\di\mu\nu$ iff $a\in\nu\To \di a\in\mu$ for all $a\in\c A$,
where $\mu,\nu$ are any \uf s of $\c A$ and $\di\in\M$,

\item its \emph{canonical extension} $\c A^\sigma$ (which is also an $\M$-BAO).
It is also known as
a `perfect extension' or  the `canonical embedding algebra',
and written $\gc{Em}\c A$.

\end{itemize}
For details,  see, e.g., \cite{JT51} or \cite[definition 5.40]{BRV:ml}.

\subsection{Canonicity in power}\label{ss:canonicity}

\begin{definition}\label{def:main canon}
We say that a variety $V$ of  $\M$-BAOs is \emph{canonical,} or 
for emphasis, \emph{totally canonical,} if $\c A^\sigma\in V$
for every $\c A\in V$.
For a cardinal~$\kappa$,
we say that $V$ is \emph{$\kappa$-canonical}
if $\c A^\sigma\in V$
for every $\c A\in V$ that is `$\kappa$-generated' --- that is,
generated as an algebra
by a subset of $\c A$ of cardinality $\leq\kappa$.

We say that $V$ is
\emph{\nice}
 if it is either totally canonical, or not even $(|\M|+\omega)$-canonical.
\end{definition}
This definition of $\kappa$-canonicity,
or canonicity in the power $\kappa$,
 is a slight variant of ones in
\cite[p.30]{Fine75} and \cite[\S6]{Gol95:canon}.
The connection of these notions to 
canonicity of modal logics is again well known  \cite[\S5.3]{BRV:ml}:
$V$ is canonical iff $L_V$ is canonical --- ie.\
every canonical frame of $L_V$ validates $L_V$ ---
and $V$ is $\kappa$-canonical iff every canonical frame
of $L_V$ defined with $\leq\kappa$ atoms validates $L_V$.

%
%

Let us consider \nice ness in the case $\M=\{\di\}$,
where $|\M|+\omega=\omega$.
 \cite[theorem 6.1]{Gol95:canon}
 gives an example of a variety
 of $\{\di\}$-BAOs
that is finitely canonical ($n$-canonical for all finite $n$)
but not $\omega$-canonical.
Nonetheless,
as far as we are aware,
all known $\omega$-canonical varieties of $\{\di\}$-BAOs are totally canonical.
\emph{No non\nice\ varieties of $\{\di\}$-BAOs are known.}
 Fine \cite[p.30]{Fine75} asked if there are any (his question was for modal logics
 and we have read it across for varieties).
 This is asking whether 
 there is a dichotomy between totally canonical and non-$\omega$-canonical
 varieties of $\{\di\}$-BAOs, with no other kinds existing.

Let us unpack this a little more.
For arbitrary $\M$ now,
let $E$ be the set of all $\M$-BAO equations,
and let $\Xi=\wp(E)$.
Varieties are defined by members of $\Xi$.
A subset $\Sigma\subseteq \Xi$
therefore defines a particular kind of variety:
a \emph{$\Sigma$-variety} is a variety of $\M$-BAOs
defined by a set of equations in $\Sigma$.
Since $\Sigma$ is a set, the ZFC axiom of replacement ensures that
there exists a cardinal $\kappa$
such that every non-canonical $\Sigma$-variety $V$
contains a $\kappa$-generated algebra $\c A$ with $\c A^\sigma\notin V$.
Hence, every $\kappa$-canonical $\Sigma$-variety
 is totally canonical.
 Let $\kappa_\Sigma$ be the least such $\kappa$.
 Then
 $\Sigma\subseteq\Sigma'\subseteq\Xi$ implies $\kappa_\Sigma\leq\kappa_{\Sigma'}$;
if  every $\Sigma$-variety is canonical then $\kappa_\Sigma=0$;
 and every $\Sigma$-variety is \nice\ iff $\kappa_\Sigma\leq|\M|+\omega$.
 
 As far as we are aware, all known varieties
 of $\M$-BAOs are \nice, and it is an open question
 whether \emph{every} variety of $\M$-BAOs is \nice\ ---
 that is, $\kappa_\Xi\leq|\M|+\omega$.
In the absence of an answer, we can still try to identify 
 large sets $\Sigma$ for which every $\Sigma$-variety is \nice.
That is what we will do in the current paper.

\section{Strategy for canonicity}\label{sec:setup}

In this section (in \S\ref{ss:strategy} and theorem~\ref{thm:truthlemma} especially)
we present a model-theoretic strategy for proving that 
a variety of $\M$-BAOs is \nice.
Along the way, we accumulate some other results
(such as  corollaries~\ref{cor:A case} and~\ref{cor:f emb})
that will be useful later.

\subsection{The variety $V$}\label{ss:var V}

We now fix a variety $V$ of  $\M$-BAOs. 
Further conditions on $V$ will be imposed later.

\subsection{The two-sorted \str\ $\gc B$}\label{ss:B}

Next we fix a `big' two-sorted \str\
$$
\gc B=(\c B,\c B_+),
$$
where $\c B\in V$.
We say that $\c B$ comprises the elements of $\gc B$ of \emph{algebra sort},
and $\c B_+$ comprises the elements of \emph{point sort.}
Further conditions on $\gc B$ will be imposed later.
It may be confusing but has to be accepted
that each \uf\ of $\c B$ is both  a {subset} of $\c B$
and an {element} of $\c B_+$.

The two-sorted signature $\c L$ of $\gc B$
comprises the following symbols with the following interpretations in $\gc B$:
\begin{enumerate}
\item The algebra operations $+,-,0,1,\di$ of $\c B$ for each $\di\in\M$, all taking algebra
elements to algebra elements. Their interpretation in $\gc B$ is 
copied from $\c B$.
As abbreviations, we write $x\cdot y=-(-x+-y)$
and $\bo x=-\di{-}x$ for each $\di\in\M$.

\item A binary relation symbol ${\inn}:\c B\times\c B_+$, with $\gc B\models b\inn \mu$ iff $b$ 
is a member of the \uf~$\mu$.

\item The binary relation symbols
$R_\di:\c B_+\times\c B_+$ (for each $\di\in \M$)
 introduced for frames in definition~\ref{def:frame},
with interpretations copied from the
canonical frame $\c B_+$.
So for each $\di\in\M$,
\[
\gc B\models\forall xy(R_\di xy\leftrightarrow\forall z(z\inn y\to\di z\inn x)).
\]
Note that formulas of the signature of frames,
such as $R_sxy$ 
($s\in{}^{<\omega}\M$) from definition~\ref{def:etc},
 are also  $\c L$-formulas.

\item Unary relation symbols $Q_i$ ($i<\omega$) of point sort.
They can be forgotten about until \S\ref{ss:strategy}.
For now, bear in mind that
they are indeed in $\c L$ and get interpretations in $\c L$-\str s.

\end{enumerate}
We leave the sorts of variables and \str\ elements to be determined by context.

\subsection{Countable $\M$}\label{ss:M ctble}

From now on, \emph{we assume that $\M$ is countable.}
This is the most common case encountered in practice,
although there are some exceptions, 
such as the signatures of some infinite-dimension\-al
cylindric and polyadic algebras.
For countable $\M$, \emph{saying that a variety is \nice\ 
says that it is canonical iff it is $\omega$-canonical.}

We make this countability restriction purely for simplicity of exposition.
Suitably formulated, our results readily generalise to uncountable $\M$:
see \S\ref{ss:final rmks}.

\subsection{The \str\ $\gc A\preceq\gc B$}\label{ss:A<B}
We fix a countable elementary sub\str\
\[
\gc A=(\c A,W)\preceq\gc B.
\]
Such a \str\ exists because $\M$, and hence $\c L$, are countable
(and this is the only time we use the countability of $\M$).
We write $\c L(\gc A)$ for $\c L$ expanded by a new constant $\u a$
for each element $a$ of $\gc A$. 
From now on, $\gc{A,B}$ will denote
 the expansions of these respective \str s to $\c L(\gc A)$-\str s in which
each $\u a$ names $a$.
Note that these expansions continue to satisfy $\gc A\preceq\gc B$.

Easily, $\gc A\preceq\gc B$ implies that 
$\c A\preceq\c B$ as algebras, so  $\c A$ is an $\M$-BAO  and $\c A\in V$
(since $V$ is an elementary class).
Although $\dom\gc A\subseteq\dom \gc B$, for clarity
we will sometimes use the inclusion map
\[
\iota:\dom\gc A\to\dom \gc B.
\]
It is an $\c L(\gc A)$-elementary embedding.

To save time, we will use the easy half of 
the \emph{Tarski--Vaught criterion}
(see \cite[2.5.1]{Hodg93} or \cite[3.1.2]{ChK90}): since
 $\gc A\preceq\gc B$,
if $\varphi(x)$ is an $\c L(\gc A)$-formula 
and $\gc B\models\exists x\varphi(x)$,
then there is $a\in\gc A$ with $\gc B\models\varphi(a)$.


\subsection{The $\c L'$-expansions $\gc{A',B'}$}

Since $\gc A$ and $\gc B$ are elementarily equivalent $\c L(\gc A)$-\str s,
by Frayne's theorem \cite[4.3.13]{ChK90}
there exist an \upwr\ $\gc A^*$
and an $\c L(\gc A)$-elementary embedding
$\sigma:\gc B\to\gc A^*$.
(The \upwr\ $\gc M^*$, using the same \uf, is defined for every \str\ $\gc M$.)
Let $\delta:\gc A\to\gc A^*$ be the diagonal embedding
(called the `natural embedding' in \cite[p.221]{ChK90}).
Because each element of $\gc A$ is named by a constant of $\c L(\gc A)$, we have
$\delta=\sigma\circ\iota$:
see figure~\ref{fig1}.

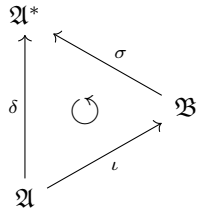
\begin{figure}[ht]
\begin{center}
\begin{tikzcd}[column sep=41pt, row sep = 20pt]

\gc A^*
\arrow[dd, phantom, bend left=55,  "\circlearrowleft" scale=1.5]

\\

&\gc B
\arrow[lu, "\sigma"']

\\

\gc A
\arrow[uu,"\delta"]
\arrow[ur,"\iota"']

\end{tikzcd}
\end{center}
\vskip-12pt

\caption{the $\c L(\gc A)$-elementary maps $\delta,\iota,\sigma$.
The diagram commutes.}\label{fig1}
\end{figure}

\begin{definition}

\begin{enumerate}
\item Let $\c L'$ denote the expansion of $\c L(\gc A)$
by a unary relation symbol $\u\mu$ of algebra sort 
for each \uf\ $\mu\in\c A_+$.

\item 
Expand $\gc A$ to an $\c L'$-\str\ $\gc A'$,
by interpreting each $\u\mu$ as
$\mu$.
In other words, for each $a\in\c A$ we put
$\gc A'\models\u\mu(a)$ iff $a\in\mu$.

\item $\gc A'^*$ denotes the $^*$-\upwr\
 $(\gc A')^*$ of $\gc A'$.
 It is an $\c L'$-expansion of $\gc A^*$.

\item 
Expand $\gc B$ to an $\c L'$-\str\ $\gc B'$,
by setting
$\gc B'\models\u\mu(b)$ iff $\gc A'^*\models\u\mu(\sigma(b))$,
for each $b\in\c B$ and $\mu\in\c A_+$.



\end{enumerate}

\end{definition}
Of course,  $\delta:\gc A'\to\gc A'^*$
is still the diagonal embedding,
so is $\c L'$-elementary.
The maps $\iota:\gc A'\to\gc B'$ and $\sigma:\gc B'\to\gc A'^*$
are $\c L'$-embeddings, but
as example~\ref{eg:conv fails} below shows,
they are not always $\c L'$-elementary.

\subsection{Simple formulas}

\begin{definition}\label{def:simple}
An $\c L'$-formula
is said to be \emph{simple} if 
it is a boolean combination of
$\c L(\gc A)$-formulas and atomic $\c L'$-formulas of the form $\u\mu(\tau)$,
where  $\mu\in\c A_+$ and $\tau$ is an $\c L(\gc A)$-term (equivalently, an $\c L'$-term)
of algebra sort.
\end{definition}

The following is a weak (but still useful) preservation result.

\begin{proposition}\label{prop:better still}
\begin{enumerate}
\item For each simple $\c L'$-formula $\theta(\b x)$
and each tuple $\b b$ of elements of $\gc B$ whose sorts match those of $\b x$,
we have
 $\gc B'\models\theta(\b b)$ iff $\gc A'^*\models\theta(\sigma(\b b))$.

\item For each simple $\c L'$-formula $\theta(\b x,\b y)$,
if $\gc A'\models\exists\b x\forall\b y\theta$
then $\gc B'\models\exists\b x\forall\b y\theta$.
\end{enumerate}

\end{proposition}

\begin{proof}
(1) is proved by induction on the \str\ of $\theta$ as a simple formula.
If $\theta$ is an $\c L(\gc A)$-formula, the result holds because $\sigma$ is 
$\c L(\gc A)$-elementary.
Suppose that $\theta=\u\mu(\tau(\b x))$,
where $\mu\in\c A_+$ and
$\tau(\b x)$ is an algebra-sorted $\c L(\gc A)$-term.
Let $\b b\in\gc B$ have sorts matching those of $\b x$,\footnote{That is,
 $\b b\in\c B\strut$, since $\tau(\b x)$ is an $\c L'$-term of algebra sort, so
all variables in $\b x$ also have algebra sort.
This is because no function symbols in $\c L'$ 
take any arguments of point sort.}
and write $b=\tau^{\gc B}(\b b)\in\c B$.
Then
\[
\begin{array}{rcll}
\gc B'\models\theta(\b b)
&\iff&
\gc B'\models\u\mu(\tau(\b b))
&\mbox{as }\theta=\u\mu(\tau(\b x))
\\
&\iff&
\gc B'\models\u\mu(b)
&\mbox{by definition of  }b
\\
&\iff&
\gc A'^*\models\u\mu(\sigma(b))
&\mbox{by definition of the expansion }\gc B'
\\
&\iff&
\gc A'^*\models\u\mu(\tau(\sigma(\b b)))
&\mbox{as $\sigma$ is $\c L(\gc A)$-elementary, so }
\sigma(b)=\tau^{\gc A^*}(\sigma(\b b))
\\
&\iff&
\gc A'^*\models\theta(\sigma(\b b))
&\mbox{as }\theta=\u\mu(\tau(\b x)).

\end{array}
\]
The boolean cases are standard.

\smallskip

It is enough to prove (2)
for $\c L'$-sentences of the form 
$\forall\b y\theta$, where $\theta(\b y)$ is simple.
For, given this much, if $\gc A'\models\exists\b x\forall\b y\theta(\b x,\b y)$,
then choose $\b a\in\gc A'$
with $\gc A'\models\forall\b y\theta(\b a,\b y)$.
Let $\u{\b a}$ be the tuple of constants
in $\c L(\gc A)$ corresponding to $\b a$.
Clearly, $\theta(\u{\b a},\b y)$ is also a
simple $\c L'$-formula,
and $\gc A'\models\forall\b y\theta(\u{\b a},\b y)$.
So by assumption,
$\gc B'\models\forall\b y\theta(\u{\b a},\b y)$.
Then certainly,
$\gc B'\models\exists\b x\forall\b y\theta(\b x,\b y)$.

So let $\theta(\b y)$ be simple
and suppose  that $\gc A'\models\forall\b y\theta$.
As $\delta:\gc A'\to\gc A'^*$ is $\c L'$-elementary,
$\gc A'^*\models\forall\b y\theta$ as well.
Let $\b b\in\gc B$ be arbitrary subject to matching the sorts of $\b y$.
Since $\gc A'^*\models\forall\b y\theta$, 
certainly $\gc A'^*\models\theta(\sigma(\b b))$.
Part~1 now gives us
 $\gc B'\models\theta(\b b)$.
Since $\b b$ was arbitrary, we
obtain $\gc B'\models\forall\b y\theta$ as required.
\end{proof}

\begin{corollary}\label{cor:A case}
If $\theta(\b x)$ is a simple $\c L'$-formula,
then $\gc A'\models\forall\b x\theta$ iff $\gc B'\models\forall\b x\theta$.
\end{corollary}

\begin{proof}
$\To$  is a special 
case of proposition~\ref{prop:better still}(2).
For $\Leftarrow$, if $\gc A'\not\models\forall\b x\theta$
then $\gc A'\models\exists\b x\neg\theta$, and 
since $\neg\theta$ is also simple,
the proposition yields $\gc B'\models\exists\b x\neg\theta$. Hence,
$\gc B'\not\models\forall\b x\theta$.
\end{proof}

\subsection{Some applications}

Corollary~\ref{cor:A case} will be used several times,
 beginning with the following:

\begin{corollary}\label{cor: muB uf}
For each $\mu\in\c A_+$,
the set $\u\mu^{\gc B'}$ is an \uf\ of $\c B$.
\end{corollary}

\begin{proof}
The conclusion says that
$\gc B'\models\forall xyz(\u\mu(x)\wedge\u\mu(y)\wedge z\geq x\cdot y\to\u\mu(z))$
and $\gc B'\models\forall x(\u\mu(-x)\leftrightarrow\neg\u\mu(x))$.
These are  universally-quantified simple $\c L'$-sentences, and they hold in $\gc A'$
since $\mu$ is an \uf\ of $\c A$.
By corollary~\ref{cor:A case}, they also hold in $\gc B'$.
\end{proof}

So  $\u\mu^{\gc B'}\in\c B_+$.

\begin{definition}\label{def:W}
For $w\in W$, let $\hat w=\{a\in\c A:\gc A\models a\inn w\}$.
Let $\wh W=\{\hat w:w\in W\}$.
\end{definition}

Because $\gc B=(\c B,\c B_+)$,
we have
$\gc B\models\forall xyzt(x\inn t\wedge y\inn t\wedge z\geq x\cdot y\to z\inn t)$
and $\gc B\models\forall xt(-x\inn t\leftrightarrow \neg(x\inn t))$.
Because $\gc A\preceq\gc B$,
these $\c L$-sentences also hold in $\gc A=(\c A,W)$.
We conclude that
each $\hat w$ is an \uf\ of $\c A$,
and $\wh W\subseteq\c A_+$.

\begin{example}\label{eg:conv fails}\rm
The converse implication in 
proposition~\ref{prop:better still}(2) can fail.
For example, suppose that $\c A_+$ is uncountable.
So there is $\mu\in\c A_+\setminus\wh W$.
Consider the simple $\c L'$-formula 
\[
\theta(x,y)\;\:=\;
y\inn x\leftrightarrow\u\mu(y).
\]
By corollary~\ref{cor: muB uf}
we have $\u\mu^{\gc B'}\in\c B_+$,
and evidently
$\gc B'\models\forall y(y\inn \u\mu^{\gc B'}\leftrightarrow\u\mu(y))$.
Hence,  $\u\mu^{\gc B'}$ witnesses $\gc B'\models \exists x  \forall y\theta$.
But $\mu\notin\wh W$, so
there is no $w\in W$ with $\gc A'\models\forall y(y\inn w\leftrightarrow\u\mu(y))$.
Hence,
$\gc A'\not\models\exists x  \forall y\theta$.

So  the elementarily-equivalent $\c L'$-\str s 
$\gc A'$ and $\gc A'^*$ need not be elementarily equivalent to $\gc B'$,
since they may disagree on the $\c L'$-sentence $\exists x\forall y\theta$.
It follows that the embeddings $\iota:\gc A'\to\gc B'$
and $\sigma:\gc B'\to\gc A'^*$ are not $\c L'$-elementary in general.
\end{example}

\subsection{An embedding $f:\c A_+\to\c B_+$}\label{ss:def f}

\begin{definition}\label{def:f}
Let $f:\c A_+\to\c B_+$
be given by:
$f(\mu)=\u\mu^{\gc B'}$, for each $\mu\in\c A_+$.
\end{definition} 
By corollary~\ref{cor: muB uf}, $f$ is well defined.
We plainly have
\begin{equation}\label{e:def f}
\gc B'\models\forall x(\u\mu(x)\leftrightarrow x\inn f(\mu))
\end{equation}
for each $\mu\in\c A_+$.
A  similar map was given by Surendonk in \cite{Surendonk96};
see  \cite{Surendonk97,Surendonk98} for more on this topic.

\penalty-100

\begin{corollary}\label{cor:f emb}
The map $f$
is  a frame embedding,
and $f(\hat w)=w$
for each $w\in W$.
\end{corollary}

\begin{proof}
First we check that $f$ is injective.
Let  $\mu,\nu\in\c A_+$ and suppose that
$f(\mu)= f(\nu)$, so $\u\mu^{\gc B'}=\u\nu^{\gc B'}$.
So if $\theta$ is the $\c L'$-sentence $\forall x(\u\mu(x)\leftrightarrow\u\nu(x))$,
then $\gc B'\models\theta$.
By corollary~\ref{cor:A case} we have
$\gc A'\models\theta$ too, so 
$\mu=\nu$.

Now let $\di\in\M$  be given.
Then for each  $\mu,\nu\in\c A_+$, we have $\c A_+\models R_\di\mu\nu$
iff $\gc A'\models\forall x(\u\nu(x)\to\u\mu(\di x))$.
By corollary~\ref{cor:A case}, this is
iff $\gc B'\models\forall x(\u\nu(x)\to\u\mu(\di x))$,
iff $\c B_+\models R_\di(\u\mu^{\gc B'},\u\nu^{\gc B'})$
--- that is, $\c B_+\models R_\di(f(\mu),f(\nu))$.
So $f$ preserves each $R_\di$ both ways.

Hence, $f$ is an embedding (in the usual model-theoretic sense) from $\c A_+$ to $\c B_+$.

Finally, let $w\in W$ be given.
Let $\mu=\hat w\in\c A_+$.
Plainly, $\gc A'\models\forall x(\u\mu(x)\leftrightarrow x\inn \u w)$.
By corollary~\ref{cor:A case},
$\gc B'\models\forall x(\u\mu(x)\leftrightarrow x\inn\u w)$ as well.
By~\eqref{e:def f},
we obtain $\gc B\models\forall x(x\inn\u w\leftrightarrow x\inn f(\mu))$.
But it follows from the definition of $\gc B=(\c B,\c B_+)$
that $\gc B\models\forall yz(\forall x(x\inn y\leftrightarrow x\inn z)\to y=z)$,
so $w=f(\mu)=f(\hat w)$.
\end{proof}

We remark that $f$ need not be a bounded morphism.
See \cite[\S6]{Surendonk96} for related discussion.

\begin{problem}\label{prob:f elem}
\rm
Is $f$ elementary --- or can it be made so by  choosing the \upwr\ $\gc A^*$ appropriately?
\end{problem}


\subsection{A strategy for proving total canonicity}\label{ss:strategy}

We will now describe a method, 
based on embeddings $f:\c A_+\to\c B_+$ as above, 
that will be used to show that certain varieties are \nice.

We  assume that  $\phi_0$ 
is a modal $\M$-formula true at some point in $\c B_+$ under the
valuation $g:\Q\to\wp(\c B_+)$
given by $g(q_i)=Q_i^\gc B$ for each $i<\omega$.

Recall that each modal $\M$-formula $\psi$ 
has a \emph{standard translation:}
an $\c L$-formula
$\psi^x(x)$ for each point-sorted variable $x$.
We define $\psi^x$ by induction on $\psi$ as usual: 
$\top^x=\top$,
$q_i^x=Q_i(x)$ for $i<\omega$, 
$(\neg\psi)^x=\neg(\psi^x)$,
$(\psi\vee\chi)^x=\psi^x\vee\chi^x$, and
$(\di\psi)^x=\exists y(R_\di xy\wedge\psi^y)$
for each $\di\in \M$, where $y$ is some point-sorted variable other than $ x$.
A standard induction on $\psi$
now shows that for each $t\in\c B_+$ (and any $x$),
we have
\begin{equation}\label{e:stdtr}
(\c B_+,g),t\models\psi \iff \gc B\models\psi^x(t).
\end{equation} 

Since $\phi_0$ is satisfied in the Kripke model $(\c B_+,g)$,
we have $\gc B\models\exists x\phi_0^x$.
By the Tarski--Vaught criterion (\S\ref{ss:A<B}), 
there is $w_0\in W$
with $\gc B\models\phi_0^x(w_0)$,
and so
\begin{equation}\label{e:phi0}
(\c B_+,g),w_0\models\phi_0.
\end{equation}
We fix such a $w_0$.
Let $\c F=\c A_+(\wh{w_0})$ be
the inner subframe of $\c A_+$ generated by $\wh{w_0}$
(see definitions~\ref{def:etc} and~\ref{def:W}),
and let $\c M$ be the
submodel of $(\c B_+,g)$ with domain $f(\c F)$.
By corollary~\ref{cor:f emb}, $w_0=f(\wh{w_0})\in f(\c F)=\c M$.

\begin{definition}\label{def:truth lem}
We will say that \emph{the Truth Lemma holds in $V$} if the following always holds in these given circumstances:

\begin{quote}\em
For each $m\in\c M$ and modal $\M$-formula $\psi$,
we have $\c M,m\models\psi$ iff $(\c B_+,g),m\models\psi$.
In particular, $\c M,w_0\models\phi_0$.
\end{quote}
\end{definition}

\begin{theorem}\label{thm:truthlemma}
If the Truth Lemma holds in $V$,
then $V$ is \nice.
\end{theorem}

\begin{proof}
Assume that $V$ is not canonical.
So there is $\c B\in V$ with  $\c B^\sigma\notin V$.
Hence, some equation $\varepsilon$ valid in $V$
is not valid in $\c B^\sigma$.
Translate $\varepsilon$ to a modal formula $\varphi_\varepsilon$ 
as in \S\ref{ss:algs}.
Then $\neg\varphi_\varepsilon$ is satisfiable in the frame $\c B_+$.
Form $\gc B$ from $\c B$, 
choosing the $Q_i^\gc B$ so 
that the valuation $g$ above will satisfy $\neg\varphi_\varepsilon$ in $\c B_+$,
and then form $\gc A$, $\c A$, $W$,
$w_0$, $\c F$, $\c M$,  etc.,
all as above, using
 $\phi_0=\neg\varphi_\varepsilon$.
By the Truth Lemma, $\c M,w_0\models\neg\varphi_\varepsilon$.

Now the frame $f(\c F)$ of $\c M$ is isomorphic (via $f^{-1}$) to
 $\c F$.
Hence,
 $\neg\varphi_\varepsilon$ is satisfiable in $\c F$.
As $\c F$ is an inner subframe of $\c A_+$, 
standard modal arguments (see, e.g.,
\cite[theorem 3.14]{BRV:ml} or \cite[corollary 2.9]{ChagZak:ml})
show that
$\neg\varphi_\varepsilon$ is satisfiable in $\c A_+$,
and it follows that $\c A^\sigma\not\models\varepsilon$ and $\c A^\sigma\notin V$.
Since $\c A\in V$ is countable, we see that $V$ is not $\omega$-canonical.
Hence, $V$ is \nice.
\end{proof}
For countable $\M$,
this result reduces the problem of showing that 
a variety is \nice\
to the problem of showing that the Truth Lemma holds in it.

\section{Logics above $K5_2$}\label{sec: K5}

Application of the strategy of the last section
(\S\ref{ss:strategy}) will now be illustrated by showing that it can be successfully applied in the monomodal language
(where $\M=\{\di\}$) to all normal modal logics that include a logic we call $K5_2$. To define this, for each $n<\omega$ let         $\di^n$ denote a sequence of $\di$'s of length $n$. Then the  axiom   $\di^n p\to\bo\di p$ will be called $5_n$. To describe its corresponding frame condition we introduce the notation $R_\di{}^n$ for the relation $R_{\di^n}$, i.e.\ the $n$-fold composition of $R_\di$ with itself. A frame validates  $5_n$  iff its relation  satisfies
\begin{equation}  \label{n152}
 \text{for all $x,y,z$, \ $R_\di{}xy$ and  $R_\di{}^{n}xz$  implies  $R_\di yz$.}
  \end{equation} 
 $5_n$ is a Sahlqvist formula with first-order correspondent \eqref{n152}. This implies that if an algebra $\c B$ validates $5_n$, then its canonical frame $\c B_+$ will satisfy \eqref{n152} and so validate $5_n$.
  
$5_0$ is the Brouwerian axiom $ p\to\bo\di p$, corresponding to symmetry of $R_\di$, which is  the case $n=0$ of \eqref{n152}.
  $5_1$ is $\di p\to\bo\di p$, whose equivalent form $\di \bo p\to \bo p$ is commonly known as 5, due to its role in the definition of Lewis's system S5.

\subsection{Canonicity of extensions of $K5_2$} \label{ss:K52 canon}
We now turn to $5_2$.

\begin{lemma}   \label{lem:valid52}
A frame validates  $5_2$ iff for all $n<\omega$  it satisfies
\begin{equation}  \label{e52}
 \text{for all $x,y,z$, \  $R_\di{}^nxy$ and  $R_\di{}^{n+1}xz$  implies  $R_\di yz$.}
  \end{equation} 
\end{lemma}

\begin{proof}
The case $n=1$ of \eqref{e52} is the same as the case $n=2$ of \eqref{n152}, which is the condition for validity of $5_2$. So it suffices to show that the case $n=1$  of \eqref{e52} implies that  \eqref{e52} holds for all $n$.    

So assume the $n=1$ case of \eqref{e52}. Then we prove \eqref{e52} by induction on $n$. The base case $n=0$ asserts that if $x=y$ and $R_\di xz$ then $R_\di yz$, which is true. Now assume inductively that \eqref{e52} holds for $n$. Then if $R_\di{}^{n+1}xy$ and  $R_\di{}^{n+2}xz$, there are $y',z'$ such that $R_\di{}^nxy'$ and $R_\di{}y'y$, and $R_\di{}^{n+1}xz'$ and $R_\di{}z'z$. From $R_\di{}^nxy'$ and $R_\di{}^{n+1}xz'$ by the hypothesis on $n$ we get $R_\di{}y'z'$. Thus we have $R_\di{}y'y$ and $R_\di{}^2y'z$ (since $R_\di{}y'z'$ and $R_\di{}z'z$), hence $R_\di{}yz$ by the $n=1$ case. Thus \eqref{e52} holds for $n+1$.
\end{proof}

\begin{proposition}\label{prop:tl for 52}
Let $V$ be any variety of monomodal 
BAOs that validate $5_2$. Then the Truth Lemma holds in $V$.
\end{proposition}

\begin{proof}
Suppose we have an algebra $\c B\in V$, a countable elementary substructure $\gc A=(\c A,W)$ of $\gc B=(\c B,\c B_+)$, an embedding $f:\c A_+\to\c B_+$
such that $f(\hat w)=w$ for every $w\in W$ (corollary~\ref{cor:f emb}),
and a valuation $g$ on $\c B_+$; with $\c M$ being the submodel of $(\c B_+,g)$ with 
domain $f(\c F)$, where $\c F$ is the inner subframe $\c A_+(\wh{w_0})$ of $\c A_+$ 
generated by  $\wh{w_0}$ for some $w_0\in W$.
Then we have to show that for each $m\in\c M$ and modal formula $\psi$,
we have $\c M,m\models\psi$ iff $(\c B_+,g),m\models\psi$.

The proof is by induction on the formation of $\psi$. The significant case is $\di\psi$, assuming inductively the result for $\psi$.
If $\c M,m\models\di\psi$, then there is some $z_0$ in $\c M$ with $R_\di^{\c M} m z_0$ and 
$\c M,z_0\models\psi$. But then  $(\c B_+,g),z_0\models\psi$ by the induction hypothesis on $\psi$, and
$R_\di^{\gc B} m z_0$ as $\c M$ is a submodel of $\c B_+$. This shows that $(\c B_+,g),m\models\di\psi$.

Conversely, suppose that
$(\c B_+,g),m\models\di\psi$ with $m\in\c M$.
Then there is some $z_0\in\c B_+$ with $R_\di^{\gc B} m z_0$  and  $(\c B_+,g),z_0\models\psi$. Hence $\gc B\models\psi^x(z_0)$ by \eqref{e:stdtr}. Since $m\in\c M$,  we have $(R_\di^{\gc B})^n w_0 m$ for some $n\geq 0$, hence  $(R_\di^{\gc B})^{n+1}w_0 z_0$.
Therefore
\begin{equation*}
\gc B\models \exists z_1\cdots\exists z_n\exists x(R_\di w_0z_1
\land R_\di z_1z_2\land\cdots\land R_\di z_nx\land\psi^x).
\end{equation*}
Now $w_0\in W$, so $n+1$ applications of the Tarski--Vaught criterion to this last fact yield elements $w_1,\dots,w_n,w\in W$ such that 
$$
\gc B\models R_\di w_0 w_1\land R_\di w_1w_2\land\cdots\land R_\di w_nw\land\psi^x(w).
$$
So $(R_\di^{\gc B})^{n+1}w_0w$.
Now $W\not\subseteq \c M$ in general.
But as $f:\c A_+\to\c B_+$ is an embedding, and  $\hat u\in\c A_+$ and $f(\hat u)=u$
for each $u\in W$,
we see that $(R^{\c A_+}_\di)^{n+1}\wh{w_0}\hat w$.
Since $\c F=\c A_+(\wh{w_0})$,
we obtain $\hat w\in\c F$,
and hence $w=f(\hat w)\in\c M$.
By \eqref{e:stdtr}, $(\c B_+,g),w\models\psi$. 
Hence $\c M,w\models\psi$ by the induction hypothesis.
Also, 
we now have  $(R_\di^{\gc B})^n w_0 m$ and 
$(R_\di^{\gc B})^{n+1} w_0 w$, so  \eqref{e52} gives $R^{\gc B}_\di mw$.
This is because $\c B\in V$, so $\c B$ validates $5_2$, hence so does the frame $\c B_+$, and so Lemma \ref{lem:valid52} applies to $\c B_+$.

Since $R^{\gc B}_\di mw$ we get  $R^{\c M}_\di mw$ as $m$ and $w$ are in $\c M$. Together with $\c M,w\models\psi$, this implies the desired  $\c M,m\models\di\psi$, and completes the inductive case for $\di$.
\end{proof}

$K5_n$ is the smallest normal logic to contain $5_n$. 
$K5_2$  is a sublogic of $K5$, as can be seen model-theoretically by deriving the case $n=2$ of  \eqref{n152} from its $n=1$ case. (A proof-theoretic demonstration is also straightforward.)
It follows from the result just proved and Theorem \ref{thm:truthlemma} that:

\begin{theorem}\label{thm:K52 canon}
Every normal logic extending $K5_2$ that is $\omega$-canonical is totally canonical. 
\end{theorem}

\subsection{Continuum-many logics between $K5_2$ and $K5$}\label{ss:continuum K52-K5}

Theorem~\ref{thm:K52 canon}
 applies to all normal extensions of $K5$. But whereas $K5$ has only countably many extensions \cite{nagl:exte85}, we will now show that  $K5_2$ has continuum many. Indeed there are continuum many between $K5_2$ and $K5$.

The proof adapts a construction from \cite[p.162]{ChagZak:ml}. 
For each $j\in\omega$, let $\c D_j$ be the `diamond-shaped' frame depicted in Figure \ref{figk52}.
\begin{figure}
\begin{center}

\tikzset{->-/.style={decoration={
  markings,
  mark=at position #1 with {\arrow{>}}},postaction={decorate}}}

\begin{tikzpicture}[ scale=1, >=Stealth, 
V/.style={},
e/.style={}]

\node (j+1) at (0,0) [V]  {$j+1$} ;
\node (j') at (2,-1.4) [V]  {$j'$} ;
\node (j'') at (2,1.4) [V]  {$j''$} ;
\node (j-1) at (4,0) [V]  {$j-1$} ;

\node (j-2) at (6,0) [V]  {$j-2$} ;
\node (etc) at (8,0) [e]  {$\cdots$} ;
\node (1) at (10,0) [V]  {$1$} ;
\node (0) at (12,0) [V]  {$0$} ;

\draw[->-=.6]  (j+1) to  (j');
\draw[->-=.6]  (j+1) to  (j'');
\draw[->-=.6]  (j') to  (j-1);
\draw[->-=.6]  (j'') to  (j-1);
\draw[->-=.7]  (j-1) to  (j-2);
\draw[->-=.6]  (j-2) to  (etc);
\draw[->-=.6]  (etc) to  (1);
\draw[->-=.6]  (1) to (0); 

\end{tikzpicture}

\caption{ The intransitive frame $\c D_j$ validating $5_2$}
\label{figk52}
\end{center}
\end{figure}
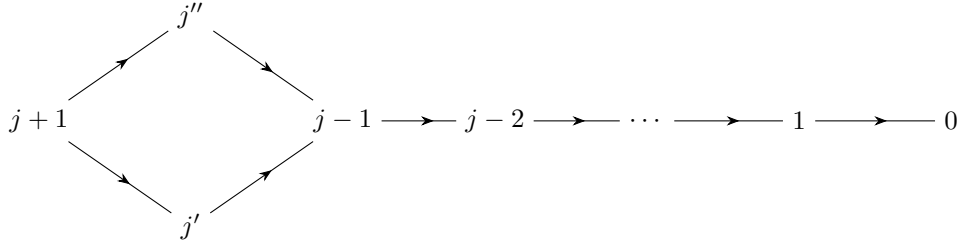
We take this as an intransitive
 irreflexive frame in which the only relations that  hold are those given by the displayed arrows. Thus the  relation $R_\di$ of $\c D_j$ is
$$
\begin{array}{rl}
\{ (j+1,j'),(j+1,j''), (j',j-1),  (j'',j-1)\} \cup\{(i,i-1): j-1\geq i\geq 1\},&\mbox{if }j>0,
\\[2pt]
\{ (1,0'),(1,0'')\},&\mbox{if }j=0.
\end{array}
$$
By inspection, this relation satisfies \eqref{n152} with $n=2$, so $\c D_j$ validates 5$_2$.
(If we replaced $R_\di$ by its transitive closure, $5_2$ would not be valid when $j>0$.)

For each $i\in\omega$, let $\alpha_i$ be the constant formula $\di^i\top\land\neg\di^{i+1}\top$. 
In $\c D_j$, if $i<j$ or $i=j+1$, then $\alpha_i$ holds at $i$ and nowhere else. If $i>j+1$, then $\alpha_i$ does not hold at any point. On the other hand, $\alpha_j$ holds at both $j'$ and $j''$. Now let $\varphi_i$ be the formula
$$
\bo(\alpha_i \to p) \lor \bo(\alpha_i \to \neg p).
$$
If $i\ne j$, then $\alpha_i$ is true at no more than one point in $\c D_j$, 
so $\c D_j\models\varphi_i$. But
 $\c D_j\not\models\varphi_j$, 
 for making $p$ true just at $j'$ gives a model on $\c D_j$ at which $\alpha_j \to p$ is false at $j''$ and
  $\alpha_j \to \neg p$ is false at $j'$, hence $\varphi_j$ is false at $j+1$.
  
  For each non-empty subset $I\subseteq \omega$, let $L_I$ be
  the smallest normal modal logic containing
 $K5_2\cup\{\varphi_i:i\in I\}$. 
 Suppose that $I\ne J\subseteq\omega$ with, say, some $j\in J\setminus I$.  
 As $j\notin I$ we have $\c D_j\models\varphi_i$ for all $i\in I$. Since also $\c D_j\models K5_2$,  the logic determined by $\c D_j$ includes   $K5_2\cup\{\varphi_i:i\in I\}$, so it includes $L_I$ as the smallest such. Thus $\c D_j\models L_I$. But $\c D_j\not\models \varphi_j$, so then $\varphi_j\notin L_I$.
  On the other hand $j\in J $, so $\varphi_j\in L_J$ by definition.  Hence $L_I\ne L_J$.  The case $I\setminus J\ne\emptyset$ is likewise. 
  
  Thus  there are $2^{\aleph_0}$  logics $L_I$, one for each subset $I$ of $\omega$, and they all include K5$_2$. But they are all included in $K5$, because every formula $\varphi_i$ is a $K5$-theorem. To show this, it is enough to show that $\varphi_i$ is valid in every point-generated $K5$-frame, since these frames determine the logic. By an analysis originally due to Segerberg \cite{sege:deci68a}, each such frame   either (i) is a single irreflexive point, or (ii) has the property that every point has a reflexive point in its future. In case (i), every formula of the form $\bo\psi$ is valid in the frame, hence the $\varphi_i$'s are. In case (ii), for any $i\in I$ the formula $\di^{i+1}\top$ is true at every point, hence $\alpha_i$ is false everywhere and so $\varphi_i$ again is valid.


\section{The modal logics $KU_n^\M$}\label{sec:KUn}

In the remainder of the paper,
we will define a wider
class of logics and extend the
 canonicity analysis for $K5_2$
to them. 

The rough idea is this.
Take $\M=\{\di\}$.
By~\eqref{n152} for $n=2$,
a frame $\c F$ validates $K5_2$ iff
for each $x\in\c F$, \emph{all points in $R^\c F_\di(x)$ have the same future:}
$R^\c F_\di(y)=R^\c F_\di(y')$ for all $y,y'\in R^\c F_\di(x)$.
Extended to $(R^\c F_\di)^n(x)$
by lemma~\ref{lem:valid52},
this frame condition was used
in the proof of proposition~\ref{prop:tl for 52}
for the crucial deduction
$$
(R_\di^{\gc B})^n w_0 m
\wedge (R_\di^{\gc B})^{n} w_0 w_n
\wedge R_\di^\gc B w_nw
\;\;\To\;\; R^{\gc B}_\di mw.
$$
Lemma~\ref{lem:valid52} ensures that
$R^\gc B_\di(w_n)= R^\gc B_\di(m)$, so the deduction goes through.
\emph{But the deduction can be made assuming only that
$R^\gc B_\di(w_n)\subseteq R^\gc B_\di(m)$.}
For then, $w\in R^\gc B_\di(w_n)\subseteq R^\gc B_\di(m)$, so $R^{\gc B}_\di mw$ as required.

This suggests weakening the $5_2$ frame condition to:
\emph{any two points in any $R^\c F_\di(x)$ have $\subseteq$-comparable futures.}
That is, $R^\c F_\di(y)\subseteq R^\c F_\di(y')$
or $R^\c F_\di(y')\subseteq R^\c F_\di(y)$ for all $y,y'\in R^\c F_\di(x)$.
This may seem to allow
$R^\gc B_\di(w_n)\supseteq R^\gc B_\di(m)$ above, rather than the desired converse inclusion;
we will solve this problem in proposition~\ref{prop:cof}.
We then weaken it further by replacing `two' by `$n+1$', for arbitrary $n>0$.
This is handled in the proof of proposition~\ref{prop:cof} by adding a set $K$ of parameters.
We arrive at:
\begin{itemize}
\item  \emph{given any $n+1$ points from any $R^\c F_\di(x)$,
two of them have $\subseteq$-comparable futures.}
\end{itemize}
Theorem~\ref{thm:can} will extend
theorem~\ref{thm:K52 canon}  to this case, and for arbitrary~$\M$.

\smallskip

So we are led to consider sets and sequences of points
in a Kripke frame whose futures are pairwise incomparable
with respect to inclusion,
forming a $\subseteq$-antichain.
Borrowing from the physics of spacetime,
we call such sets and sequences \emph{\ach.}
We will define them formally below, and then introduce some modal logics
called $KU^\M_n$
that limit the size of \ach\ sets,
and so stand as weakenings of $K5_2$ to which the methods of \S\ref{ss:K52 canon}
can still be applied.

We will then 
take a break from canonicity
and spend the rest of this section, and all of the next,
examining these logics.
We return to canonicity in \S\ref{sec:main canon}, where we
prove \nice ness of all normal modal $\M$-logics extending some $KU^\M_n$.
They will be called \emph{logics of finite \ach\ width.}

\subsection{Achronal sets and sequences}
\begin{definition}\label{def:pif}
Let $\c F$ be a frame and $\di\in\M$.
\begin{enumerate}

\item 
A subset $S\subseteq\c F$ is said
to be \emph{$\di$-\ach}
if
$R_\di^\c F(x)\not\subseteq R_\di^\c F(y)$ for each distinct $x,y\in S$.

\item Let  $\alpha$ be an ordinal.
A sequence $(y_i:i<\alpha)$ of points of $\c F$
is said to be \emph{$\di$-\ach}
if $R_\di^\c F(y_i)\not\subseteq R_\di^\c F(y_j)$ for each distinct $i,j<\alpha$.

\item An  \emph{\ach\ set (or sequence)}
is a $\di$-\ach\ set (or sequence, resp.)\ for some $\di\in\M$.

\end{enumerate}
\end{definition}
Sets with at most one element are $\di$-\ach.
Subsets of $\di$-\ach\ sets are $\di$-\ach, and similarly for sequences.
The following lemma
is elementary but worth nailing down.

\begin{lemma}\label{lem:pif}
Let $\c F$ be a frame, $S\subseteq\c F$,
$\di\in\M$, and  $\kappa$  a cardinal.
The following are equivalent:
\begin{enumerate}[\quad\rm 1.]
\item $S$ has a $\di$-\ach\ subset of cardinality $\kappa$,

\item there exists a $\di$-\ach\ sequence $(y_i:i<\kappa)$
of elements of $S$,

\item there is $Y\subseteq S$ such that $R^\c F_\di[Y]$
is an $\subseteq$-antichain of cardinality $\kappa$
(see definition~\ref{def:bin rels}),

\item $R^\c F_\di[S]$ has $\subseteq$-width at least $\kappa$
(see definition~\ref{def:bin rels}).

\end{enumerate}

\end{lemma}

\begin{proof}
For $1\To2$, if $Y\subseteq S$ is $\di$-\ach\ and of cardinality $\kappa$,
let $(y_i:i<\kappa)$ enumerate $Y$ without repetitions. 
 This is a $\di$-\ach\ sequence
in $S$.

For $2\To3$, let $(y_i:i<\kappa)$ be a $\di$-\ach\ sequence in $S$.
Then $Y=\{y_i:i<\kappa\}$ satisfies~3.

$3\To4$ is trivial.
For $4\To1$, assume 4.
So there is 
 an $\subseteq$-antichain  $\boldsymbol{A}\subseteq R^\c F_\di[S]$ of cardinality $\kappa$.
For each $Z\in\boldsymbol{A}$, pick a point $y_Z\in S$ with $R^\c F_\di(y_Z)=Z$.
Then $\{y_Z:Z\in\boldsymbol{A}\}\subseteq S$ is \ach\ and of cardinality $\kappa$.
\end{proof}

The condition
`there is $Y\subseteq S$ of cardinality $\kappa$
such that 
$R^\c F_\di[Y]$ is an $\subseteq$-antichain'
 clearly follows from condition~3 in the lemma,
but is strictly weaker, as we may have $|R^\c F_\di[Y]|<\kappa$.
As an informal example, let $\c F=(\omega,R)$,
where  $R=\emptyset$.
Here, $\c F$ has no \ach\ subset with more than one element,
but $R[\omega]=\{\emptyset\}$ is trivially an $\subseteq$-antichain.

\subsection{The logic $KU_n^\M$}\label{ss:Un rmks}


\begin{definition}\label{def:Un}
Let $0<n<\omega$.
We write $KU_n^\M$ for the smallest normal modal
$\M$-logic containing the following set $U^\M_n$ of axioms:
\[
U_n^\M=\Big\{
\bigvee_{i\leq n}\bo\big(\bbo q_i\to\bigvee_{j\leq n,\,j\neq i}\bbo q_j\big)
\;:\;
\di,\bdi\in\M
\Big\}.
\]
For avoidance of doubt, $i$ and $j$ here denote ordinals in $\{0,\ldots,n\}$.

A normal modal $\M$-logic is said to be
of \emph{finite \ach\ width} if it contains $U^\M_n$ for some $n$ with $0<n<\omega$.
\end{definition}


\subsection{First-order correspondents of the $U_n^\M$}\label{sss:corresps}

The $U_n^\M$ axiom shown above
is equivalent to 
$(\bigwedge_{i}\di(\bbo q_i\wedge\neg\bigvee_{j\neq i}\bbo q_j))\to\bot$,
a Sahlqvist formula according to \cite[definition 3.51]{BRV:ml}.
Its  \fo\  correspondent
\cite[theorem 3.54]{BRV:ml} is
\[
\forall xy_0\ldots y_n\bigvee_{i\leq n}
\big[ R_\di xy_i\to\bigvee_{j\leq n,\,j\neq i}\forall z(R_\bdi y_iz\to R_\bdi y_jz)\big].
\]
The part following the $\forall xy_0\ldots y_n$
is the disjunction of the formulas
$\neg R_\di xy_i$, for each $i\leq n$, 
and $\forall z(R_\bdi y_iz\to R_\bdi y_jz)$, for each distinct $i,j\leq n$.
Reordering these to put the $\neg R_\di xy_i$ first, 
we see that the whole correspondent
is logically equivalent to
\begin{equation}\label{e:fo cor}
\forall xy_0\ldots y_n
\Big(
\Big(
\bigwedge_{i\leq n}R_\di xy_i
\Big)
\to \bigvee_{i,j\leq n,\>i\neq j}
\forall z(R_\bdi y_iz\to R_\bdi y_jz)
\Big).
\end{equation}
This is the form that we will use.
A frame validates $U_n^\M$ iff \eqref{e:fo cor} holds in the frame
for every $\di,\bdi\in\M$.

Informally, \eqref{e:fo cor} says that 
\emph{for any $R_\di$-successors $\Vec yn$ of some given world,
we have $R_\bdi(y_i)\subseteq R_\bdi(y_j)$ for some distinct $i,j\leq n$.}
That is, \emph{the sequence $(y_i:i\leq n)$ is not $\bdi$-\ach.}
So by  lemma~\ref{lem:pif}, \eqref{e:fo cor}~is equivalent to each of:
\begin{itemize}

\item every $\bdi$-\ach\ set contained in some $R_\di(x)$ has at most $n$ elements,


\item for each $x$, the $\subseteq$-width of $R_\bdi[R_\di(x)]$
is at most $n$.

\end{itemize}
Consequently, $U_n^\M$ is valid in precisely those frames
$\c F$ such that for each $x\in\c F$ and $\di\in\M$,
every \ach\ subset of $R^\c F_\di(x)$
has at most $n$ elements.

\subsection{Alternative formulation}
The axiom shown in definition~\ref{def:Un} can be
reformulated more in the style of \cite{+Fine74,Xu21}, as
\[
\Big(\bigwedge_{i\leq n}\di (p_i\wedge\bbo q_i)\Big)
\to\bigvee_{i,j\leq n,\,i\neq j}\di(p_i\wedge\bbo q_j),
\]
where $\Vec pn,\Vec qn\in\Q$ are pairwise distinct atoms.
This form is also Sahlqvist, with the same \fo\ correspondent,
and is therefore provably equivalent to the earlier axiom.
We give it in case it is more appealing.

\subsection{$U_1^\M$}\label{ss:U1}
To gain some familiarity with the $U_n^\M$,
let us quickly discuss the simplest case,
when $M=\{\di\}$ and $n=1$.
Then $U_1^\M=\{\bo(\bo q_0\to\bo q_1)\vee\bo(\bo q_1\to\bo q_0)\}$.
This is not a new axiom.
For example, in~\cite[lemma 3.2]{+vb78},
van Benthem
used  $\bo(\bo p\to\bo q)\vee\bo(\bo q\to\bo p)$ as part of a
 Kripke-incompleteness proof, and stated 
 an equivalent form of the correspondent~\eqref{e:fo cor} for it.
We will say a little more about this in~\S\ref{sec:kr comp}.

Plainly, every frame $\c F$ such that 
\begin{equation}\label{e:*}
(R^\c F_\di[\c F],\subseteq)
\mbox{ is a chain
---  ie.\ }
R^\c F_\di(x)\subseteq R^\c F_\di(y)
\mbox{ or }
R^\c F_\di(y)\subseteq R^\c F_\di(x)
\mbox{ for every } x,y\in\c F
\end{equation}
has no \ach\ subset with more than one element,
so satisfies~\eqref{e:fo cor} for $n=1$
and validates $U_1^\M$.
Indeed, it validates $U^\M_n$ for all $n\geq1$.

This provides a source of simple examples.
Any frame forming a transitive chain, such as $(\omega,<)$, satisfies \eqref{e:*},
but other frames do too.
For example, the irreflexive transitive frame $\c G_j$ shown in figure~\ref{fig:continuum}
below satisfies \eqref{e:*} and  validates $U_1^\M$,
but has an infinite $R$-antichain.

\subsection{Comparing the $U_n^\M$ with each other}\label{sss:Un chain}

Let  $K^\M$ denote the smallest normal modal $\M$-logic.
Plainly, \eqref{e:fo cor} for $n$
implies \eqref{e:fo cor} for $n+1$,
and it follows that
$KU_1^\M\supseteq KU_2^\M\supseteq\cdots\supseteq K^\M$.

We show that the inclusions are strict.
For any set $X$, let $\c F_X$ be the `lawn rake-like' frame 
with domain $\{a\}\cup X$, for some point $a\notin X$, and with
\begin{equation}\label{e:fan}
R_\di^{\c F_X}=
R=\{(a,x),(x,x):x\in X\}
\quad\mbox{for each }\di\in\M.
\end{equation}
The \ach\ subsets of $\c F_X$ are $\{a\}$ and the subsets of $X$.
These all have size at most $|X|$, if $X\neq\emptyset$.
So for $0<n<\omega$,
the frame $\c F_{n+1}$
validates $U^\M_{n+1}$.
But $\c F_{n+1}$ does not validate $U^\M_n$,
since  $n+1$ is an \ach\ set
with $>n$ elements contained in $R(a)$.
So the inclusion $KU_n^\M\supset KU_{n+1}^\M$ is strict.
We will show more in \S\ref{sss:KUn+1-KUn}.

Given any  modal $\M$-formula $\varphi\notin K^\M$,
filtrating the canonical model for $K^\M$
provides a finite Kripke model $\c M$ satisfying $\neg\varphi$.
Let $|\c M|=n$, say.
Then evidently, the frame of $\c M$ validates~$U_n^\M$.
Hence, $\varphi\notin KU_n^\M$.
We deduce that $\bigcap_{0<n<\omega}KU_n^\M=K^\M$.

Of course there are many normal $\M$-logics (containing $K^\M$ but) not containing any
$KU_n^\M$.
For instance, the logic of the lawn-rake frame $\c F_\omega$
 does not contain any $KU^\M_n$,
since $\omega$ is an infinite \ach\ subset of~$\c F_\omega$.

\subsection{Comparing the $U_n^\M$ with other logics}\label{ss:other logics}

Let us take a look at  the monomodal case, where $\M=\{\di\}$.
We will write $U^{\{\di\}}_n$ as simply $U_n$,
and $R_\di$ as just $R$.
For suitable $n$,
the logic   $KU_n$ weakens some known logics.
\begin{enumerate}
\item\label{i:U,1 52 and 5}
 $KU_1\subset K5_2\subset K5$,
where $K5_2$ and $K5$
are as described in \S\ref{sec: K5}.

As we said at the start of \S\ref{sec:KUn}, 
the correspondent
of $5_2$ (see~\eqref{n152} and~\eqref{e52})
says that \emph{all successors of a given world have the same $R$-future,}
and this implies~\eqref{e:fo cor} with $n=1$.
So every frame that validates $5_2$ also validates $U_1$.

Since $K5_2$ is Sahlqvist axiomatised, it is Kripke complete
\cite[theorem 4.42]{BRV:ml}.
So for each modal formula $\varphi$, if $\varphi\notin K5_2$
then $\neg\varphi$ is satisfied in some Kripke model whose
frame validates $5_2$.
By the above, this frame also validates $U_1$, 
and consequently $\varphi\notin KU_1$.
We deduce that $KU_1\subseteq K5_2$ as stated. 
The  inclusion $K5_2\subseteq K5$
was indicated in \S\ref{ss:K52 canon}.

The frame $(\omega,<)$
validates $U_1$, as we saw in \S\ref{ss:U1},
 but any two distinct points in it have different futures, so it does not validate $5_2$.
Hence, the inclusion $KU_1\subset K5_2$ is strict.
We saw in \S\ref{ss:continuum K52-K5} that the inclusion $K5_2\subset K5$ is 
(very) strict.

\item\label{i:Un and 4In}
 $KU_n\subset K4I_n$ for every $n$, where the latter is
as in Fine's paper \cite{+Fine74} and is the smallest normal modal logic
containing the Sahlqvist axioms
\begin{equation}\label{e:In corr}
\begin{array}{rcll}
4&=&\di\di p\to\di p &\mbox{(transitivity),}
\\
I_n&=&(\bigwedge_{i\leq n}\di q_i)\to\bigvee_{i,j\leq n,\,i\neq j}\di(q_i\wedge(q_j\vee\di q_j)).
\end{array}
\end{equation}
The \fo\ correspondent of $I_n$
says that \emph{no $R(x)$ contains an $R$-antichain 
with $>n$ points.}
Given this and transitivity,
if $\Vec yn\in R(x)$,
 there are $i\neq j$ with $y_i=y_j$ or $Ry_iy_j$,
and hence (by transitivity) with $R(y_j)\subseteq R(y_i)$.
So~\eqref{e:fo cor} holds for $n$,
and hence, every frame validating $4I_n$ also validates $U_n$.
Again, $K4I_n$ is Sahlqvist and so Kripke complete,
and it follows that $KU_n\subseteq K4I_n$.

The inclusion $KU_n\subset K4I_n$ is  proper for each $n$,
since frames such as 
$
(\omega,\{(i,i+1):i<\omega\})
$
validate  $U_n$ but not even $K4$.
We will see transitive examples in \S\ref{sss:KUn-K4In gap}.

\item The lawn-rake frame $\c F_{n+1}$ from \S\ref{sss:Un chain}
validates $K4I_{n+1}$, since it is transitive and its only subset
$\{a\}\cup(n+1)$ with  $>n+1$ points is not an $R$-antichain 
--- see~\eqref{e:fan}.
But as we said, it does not validate $U_{n}$.
Hence, $KU_n\not\subseteq K4I_{n+1}$.

\item\label{i:K4 and K5}
Rather trivially,
$K4\not\subseteq K5$, since
the frame 
$(\{0,1,2\},\{(0,1)\}\cup\{1,2\}^2)$
validates $K5$ but is not transitive.
Also, $K5_2\not\subseteq K4I_1$,
since $(\omega,<)$ validates $K4I_1$ but (as we saw in item~\ref{i:U,1 52 and 5})
 not~$5_2$.

\end{enumerate}
The situation is therefore as shown in figure~\ref{fig:logics}.
\begin{figure}
\begin{center}
\begin{tikzcd}[column sep=25pt, row sep = 30pt]

K4\arrow[r, "\subset", phantom]
&\cdots\arrow[r, "\subset", phantom]
&K4I_2\arrow[r, "\subset", phantom]
&K4I_1
\\
K\arrow[u, "\subset"rotate=90, phantom]\arrow[r, "\subset", phantom]
&\cdots\arrow[r, "\subset", phantom]
&KU_2\arrow[u, "\subset"rotate=90, phantom]\arrow[r, "\subset", phantom]
&KU_1\arrow[u, "\subset"rotate=90, phantom]
\arrow[r, "\subset", phantom]
&K5_2\arrow[r, "\subset", phantom]
&K5

\end{tikzcd}
\end{center}

\caption{the logics $KU_n$, $K4I_n$, $K5_2$, and $K5$ (monomodal case)}\label{fig:logics}
\end{figure}
No inclusions can be added
that do not
already follow from the ones shown by transitivity of inclusion.
\begin{enumerate}

\setcounter{enumi}{4}

\item \label{i:xu}
$U_n$ is somewhat related to Xu's axiom $\mathsf{Wid}_n^*$ \cite[p.1178]{Xu21},
whose correspondent says that 
\emph{for every $x$,
every $R$-{antichain}
with $>n$ points in $R(x)$
contains  distinct points $y,y'$
with the same proper future} ---
that is,
$\forall z(R^\bullet yz 
\leftrightarrow R^\bullet y'z)$,
where $R^\bullet yz$ abbreviates $R yz\wedge\neg R zy$.
This clearly bears some similarity to~\eqref{e:fo cor}, 
but the two are independent even with transitivity.
The  lawn-rake frame $\c F_\omega$ from \S\ref{sss:Un chain} 
is transitive  and validates all $K4\mathsf{Wid}_n^*$, but no $U_n$.
Now let $K4U_n$ be the smallest normal modal logic containing $U_n$ and the 
transitivity axiom 4.
The transitive frame
$(W,R)$,
where  $W=\{a\}\cup(\omega\times2)$,
 $a\notin\omega\times2$,  and
$R=\{(a,(i,k)),((i,k),(j,1)):j\leq i<\omega,\; k<2\}$,
validates $K4U_1$ (by virtue of~\eqref{e:*}), and so all
$K4U_n$,
but no $\mathsf{Wid}_n^*$, since no two points in the
$R$-antichain $\omega\times\{0\}$ have the same proper successors.

\end{enumerate}

\subsection{$S4U_n=S4I_n$}\label{ss:S4In=S4Un}

There is another connection between $U_n$ and $I_n$.
For a binary relation $R$ on a set $W$,
let $\o R$ be the relation on $W$ defined
by $\o Rxy$ iff $R(y)\subseteq R(x)$.
It is an exercise to verify that  
(i) an $R$-\ach\ set (defined in the obvious way)
is the same thing as an $\o R$-antichain;
(ii) $\o R$ is reflexive and transitive;
(iii) $\o R\subseteq R$ iff $R$ is reflexive;
(iv) $R\subseteq\o R$ iff $R$ is transitive.
By (ii)--(iv), we get (v) $\setbox0\hbox{$\!R$}\wd0=7pt\overline{\overline{\box0}}=\o R$.

However, (vi) $\o{\,\cdot\,}$ is not monotonic
with respect to inclusion, even on transitive relations.
For suppose $|W|\geq2$ and $R=\{(x,x):x\in W\}$ (identity).
Then $\emptyset$ and $R$ are  transitive and $\emptyset\subseteq R$.
But $R$ is also reflexive, so by (iii) and (iv), $\o R=R$.
So $\o\emptyset=W\times W\not\subseteq\o R$.

Lastly, and getting to the point,
(vii) if $\c F=(W,R)$ is rooted\footnote{`Rooted' means that
$\c F=\c F(w)$ for some $w\in W$.
The assumption is needed for $\To$.
If $a\notin\omega\times2$,
$W=\{a\}\cup(\omega\times2)$, and 
$R=\{(a,(i,1)),\,((i,0),(i,1)):i<\omega\}$,
then $R$ is transitive and $\c F\models U_1$.
But 
$\omega\times\{0\}$ is an infinite $\o R$-antichain
contained in $\o R(a)=W$, so $\o{\c F}\not\models I_n$ for any $n$.
The frame $\o{\c F}$ is  rooted, but ${\c F}$ is not.}
 and transitive, and $\o{\c F}=(W,\o R)$,
then $\c F$ validates $U_n$ iff $\o{\c F}$ validates $I_n$, for each $0<n<\omega$.

It follows from (iii), (iv), and (vii)  that $S4U_n=S4I_n$,
where these are the smallest
normal modal logics containing
 4, the reflexivity axiom $ p\to\di p$,
 and $U_n$ (respectively,~$I_n$).

\subsection{Continuum-many logics in intervals in figure~\ref{fig:logics}}

Here we remain in the monomodal case:
$\M=\{\di\}$.
We still write $U^{\{\di\}}_n$ as  $U_n$
and $R_\di$ as  $R$.

We saw in \S\ref{ss:continuum K52-K5} that there are continuum-many
normal modal logics
between $K5_2$ and $K5$.
We will now briefly indicate
that the same holds for many
of the other intervals  in figure~\ref{fig:logics}.

\subsubsection{Continuum-many logics between $K4U_n$ and $K4I_n$}
\label{sss:KUn-K4In gap}

 Fix  $n$ with $0<n<\omega$.
Then
there are continuum-many normal monomodal logics between $K4U_n$
and $K4I_n$.

The proof is similar to the one 
in \S\ref{ss:continuum K52-K5}.
For each $j<\omega$ we
define the irreflexive transitive frame
$\c G_j=(j+1\cup\{a_l:l<\omega\},R_j)$,
where the $a_l$ are  \pdt\ and not in $j+1$, and
\[
R_j
=
\{(k,m):k<m\leq j\}
\cup
\{(0,a_l),\,(a_l,m)
\;:\;
l<\omega,\;1\leq m\leq j\}.
\]
See figure~\ref{fig:continuum}.
\begin{figure}
\begin{center}

\tikzset{->-/.style={decoration={
  markings,
  mark=at position #1 with {\arrow{>}}},postaction={decorate}}}

\begin{tikzpicture}[ scale=1, >=Stealth, 
V/.style={},
e/.style={}]

\node (0) at (0,0) [V]  {$0\;$} ;
\node (a0) at (2,1.4) [V]  {$a_0$} ;
\node (a1) at (2,.7) [V]  {$a_1$} ;
\node (a2) at (2,0) [V]  {$a_2$} ;
\node (a3) at (2,-.7) [V]  {$a_3$} ;
\node (vetc) at (2,-1.5) [e]  {\raise12pt\hbox{$\vdots$}} ;

\node (1) at (4,0) [V]  {$\;1\;$} ;
\node (2) at (6,0) [V]  {$\;2\;$} ;
\node (etc) at (8,0) [e]  {$\cdots$} ;
\node (j) at (10,0) [V]  {$\;j$} ;

\draw[->-=.6]  (0) to  (a0);
\draw[->-=.6]  (0) to  (a1);
\draw[->-=.6]  (0) to  (a2);
\draw[->-=.6]  (0) to  (a3);
\draw[->-=.6]  (0) to  (vetc);

\draw[->-=.6]  (a0) to  (1);
\draw[->-=.6]  (a1) to  (1);
\draw[->-=.6]  (a2) to  (1);
\draw[->-=.6]  (a3) to  (1);
\draw[->-=.6]  (vetc) to  (1);

\draw[->-=.6]  (1) to  (2);
\draw[->-=.6]  (2) to  (etc);
\draw[->-=.6]  (etc) to  (j);

\end{tikzpicture}
\vskip-12pt
\caption{irreflexive transitive frame $\c G_j$ 
for \S\ref{sss:KUn-K4In gap} and \S\ref{sss:KU1-K52 gap}}
\label{fig:continuum}
\end{center}
\end{figure}
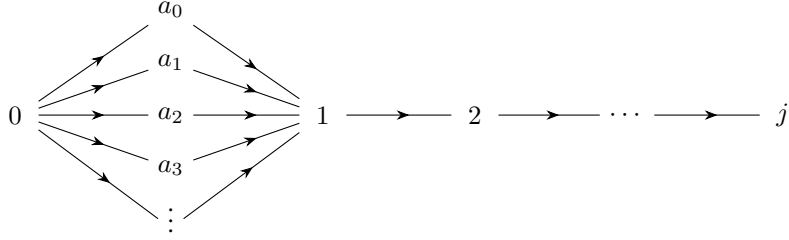
From~\eqref{e:*} in \S\ref{ss:U1},
we see
that $U_1$, and hence  $U_n$, are valid in each $\c G_j$.
Let $\c I_n$ be the class of frames that validate $K4I_n$,
and
for each $J\subseteq\omega$,
 let $L_J$ be the (normal monomodal) logic of the class 
 $\c C_J=\c I_n\cup\{\c G_j:j\in J\}$ of frames.
 For each $i<\omega$,
let $\alpha_i=\di^i\top\wedge\neg\di^{i+1}\top$
as in \S\ref{ss:continuum K52-K5},
and  
\[
\psi_i= I_n\vee\bo(\alpha_i\to p)\vee\bo(\alpha_i\to\neg p),
\]
where $I_n$ is as in~\eqref{e:In corr}
and $p$ is an atom not occurring in $I_n$.
It can be verified that
 \begin{enumerate}
\item 
$K4U_n\subseteq L_J\subseteq K4I_n$
(because $K4U_n$ is valid in $\c C_J$, and $\c I_n\subseteq\c C_J$),

\item $\psi_i$ is valid in $\c I_n$  for each $i<\omega$ (because
its disjunct $I_n$ is valid in $\c I_n$),

\item $I_n$ is not valid at 0 in any $\c G_j$
(because of the infinite $R_j$-antichain $\{a_l:l<\omega\}$),

\item $\psi_i$ is valid in $\c G_j$ iff $i\neq j$, for each $i,j<\omega$
(proved as in \S\ref{ss:continuum K52-K5}, using point 3).

\item Hence,  $\psi_i\in L_J$ iff $i\notin J$, for each $i<\omega$ and $J\subseteq\omega$.

\end{enumerate}
So if $I,J\subseteq\omega$ and 
$i\in I\setminus J$, then $\psi_i\in L_J\setminus L_I$.
We conclude that the $L_J$ ($J\subseteq\omega$)
are pairwise distinct.

\subsubsection{Continuum-many logics between $KU_1$ and $K5_2$}%
\label{sss:KU1-K52 gap}

There are also continuum-many normal modal logics
lying between $KU_1$ and $K5_2$.
Let $\c K$ be the class of frames validating $K5_2$,
and for $J\subseteq\omega\setminus\{0\}$, let $L_J$
be the logic of the class
$\c K\cup\{\c G_j:j\in J\}$,
where the $\c G_j$ are as in figure~\ref{fig:continuum} again.
For $j\geq1$, the points $a_0,1\in R_j(0)$ have different futures, so
$\c G_j\not\models 5_2$ and  $\c G_j\notin\c K$.
However,  $\c G_0\in\c K$, so we exclude $0$ from $J$.

Using 
$\xi_i=5_2\vee\bo(\alpha_i\to p)\vee\bo(\alpha_i\to\neg p)$ this time,
it can be checked that
$KU_1\subseteq L_J\subseteq K5_2$
and that the $L_J$
are \pdt.

\subsubsection{Continuum-many logics between $KU_{n+1}$ and $KU_n$}\label{sss:KUn+1-KUn}

\begin{figure}
\begin{center}

\tikzset{->-/.style={decoration={
  markings,
  mark=at position #1 with {\arrow{>}}},postaction={decorate}}}

\begin{tikzpicture}[ scale=1, >=Stealth, 
V/.style={},
e/.style={}]

\node (0) at (0,0) [V]  {$j+1$} ;

\node (a0) at (2,1.4) [V]  {$r_0$} ;
\node (a1) at (2,.7) [V]  {$r_1$} ;
\node (a2) at (2,0) [V]  {$r_2$} ;
\node (vetc) at (2,-.7) [e]  {\raise4pt\hbox{$\vdots$}} ;
\node (an) at (2,-1.4) [V]  {$r_n$} ;

\node (j'') at (4,1.4) [V]  {$j''$} ;
\node (j') at (4,-1.4) [V]  {$j'$} ;

\node (1) at (6,0) [V]  {$j-1$} ;
\node (2) at (8,0) [V]  {$j-2$} ;
\node (etc) at (10,0) [e]  {$\;\cdots\;$} ;
\node (j) at (11.5,0) [V]  {$\;1\;$} ;
\node (j+) at (13,0) [V]  {$\;0$} ;


\draw[->-=.6]  (0) to  (a0);
\draw[->-=.6]  (0) to  (a1);
\draw[->-=.6]  (0) to  (a2);
\draw[->-=.6]  (0) to  (an);

\draw[->-=.7]  (a0) to  (j');
\draw[->-=.65]  (a1) to  (j');
\draw[->-=.6]  (a2) to  (j');
\draw[->-=.6]  (an) to  (j');

\draw[->-=.6]  (a0) to  (j'');
\draw[->-=.6]  (a1) to  (j'');
\draw[->-=.6]  (a2) to  (j'');
\draw[->-=.7]  (an) to  (j'');

\draw[->-=.6]  (j') to  (1);
\draw[->-=.6]  (j'') to  (1);

\draw[->-=.6]  (1) to  (2);
\draw[->-=.6]  (2) to  (etc);
\draw[->-=.6]  (etc) to  (j);
\draw[->-=.6]  (j) to  (j+);

\end{tikzpicture}

\caption{transitive frame $\c E_j^n$ for \S\ref{sss:KUn+1-KUn}}
\label{fig:Un/Un+1}

\end{center}
\end{figure}
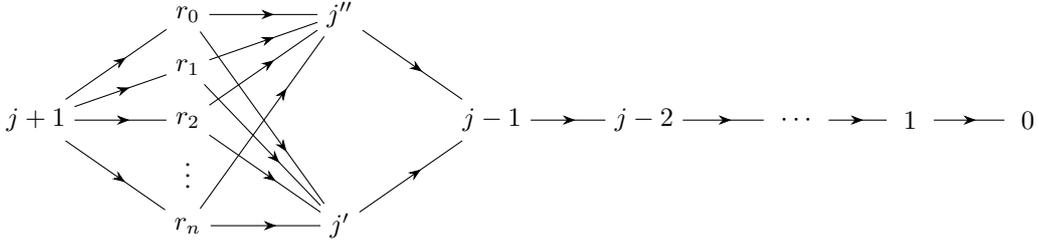

Finally, for each $0<n<\omega$,
there are continuum-many normal modal logics
 between $KU_{n+1}$ and $KU_n$. 
For $J\subseteq\omega$, let
$L_J$ be the logic
of $\c U_n\cup\{\c E_j^n:j\in J\}$,
where
$\c U_n$ is the class of frames validating $U_n$,
and $\c E_j^n$ is the transitive frame shown in figure~\ref{fig:Un/Un+1}.
This frame validates $U_{n+1}$.
What is not clear from the figure 
is that $\Vec rn$ are reflexive, the only reflexive points in the frame.
They therefore form an $(n+1)$-point \ach\ set,
and so $\c E_j^n$ does not  validate~$U_n$.

The construction of the $L_J$ this time uses
$\zeta_i=U_n\vee\bo(\alpha_i\to p)\vee\bo(\alpha_i\to\neg p)$,
noting that $\alpha_i$ is never true at $\Vec rn$ since these points are reflexive.
Again,
$KU_{n+1}\subseteq L_J\subseteq KU_n$,
and the $L_J$ are \pdt.
We leave details to the reader.

\section{Kripke incompleteness}\label{sec:kr comp}

Here we consider Kripke completeness (or as it turns out, the lack of it)  of 
finite-\ach-width logics (ie.\ normal extensions of some $KU_n^\M$).
Our starting point is the result of
Fine \cite{+Fine74} that all finite-width logics
(those extending some $K4I_n$) are Kripke complete.
We might ask whether the same holds for finite-\ach-width logics.

First, Fine's result does not
extend to multimodal logics.
For $\M=\{\di,\bdi\}$,
Thomason's normal temporal logic given in~\cite{Tho72}
(see also
\cite[theorem 4.49]{BRV:ml},
\cite[exercise 6.23]{ChagZak:ml}, and \cite[pp.55--56]{Gol87}),
extended by $I_1$ for both diamonds,
is not Kripke complete and indeed is not valid in any Kripke frame.
Since this logic contains $U_1^\M$,
not all normal extensions of even the bimodal $K4U_1^{\{\di,\bdi\}}$ are Kripke complete.

So we will restrict our attention to the monomodal case, where $\M=\{\di\}$.
As usual, we write $U^{\{\di\}}_n$ simply as  $U_n$.
In this case, we can at least say that
all normal extensions of each $S4U_n$
are Kripke complete, by Fine's result \cite{+Fine74} and
since
$S4U_n=S4I_n\supseteq K4I_n$
(see \S\ref{ss:S4In=S4Un}).

This can fail for smaller logics.
Van Benthem \cite{+vb78}
showed that the smallest normal logic
containing the axioms
$T=p\to\di p$,
$M=\bo\di p\to\di\bo p$,
$Q=\di p\wedge\bo(p\to\bo p)\to p$,
and $U_1$
is Kripke incomplete.
See also \cite[exercise 4.4.4]{BRV:ml}.
In fact, using work of Blok \cite{+blok80},
it can be shown that  $KTU_1$  has continuum-many  Kripke-incomplete
normal extensions.

\subsection{Extensions of $K4U_2$}

These latter extensions do not contain the transitivity axiom 4.
What about extensions that do?
As we will see,
the general Kripke completeness of logics containing any $K4I_n$ does not extend to those containing any $K4U_n$. We now show that there is a Kripke-incomplete normal extension of $K4U_2$. The proof builds on Fine's construction in  \cite{+Fine74c} of an incomplete logic containing S4, which we transfer to an extension of K4.

 The means of doing this is a formula translation $\phi\mapsto \phi\ci$ that is defined inductively by putting $\phi\ci=\phi$ if $\phi$ is an atom or constant, letting  the translation commute with the Boolean connectives, i.e.\ $(\neg\phi)\ci= \neg(\phi\ci)$ etc., and putting $(\di\phi)\ci=\phi\ci\lor\di(\phi\ci)$. Then $(\bo\phi)\ci=\phi\ci\land\bo(\phi\ci)$.
Introducing the  abbreviations  $\di\ci\phi$ for $\phi\lor\di\phi$,  and $\bo\ci\phi$ for $\phi\land\bo\phi$,
 we have  
$(\di\phi)\ci=\di\ci(\phi\ci)$ and $(\bo\phi)\ci=\bo\ci(\phi\ci)$.

 For any frame $\c F=(W,R)$, let  $\c F\ci=(W,R\ci)$, where  $R\ci=R\cup\{(x,x):x\in W\}$, the `reflexive closure' of $R$.
 If $\c M=(\c F,V)$ is any model on $\c F$, let   $\c M\ci=(\c F\ci,V)$.  The definition of $R\ci$ ensures that 
 \begin{equation}\label{split}
 \begin{split}
\c M,x\models \di\ci\phi &\text{\ iff $\exists y(xR\ci y$ and $\c M,y\models\phi$), and}
 \\
\c M,x\models \bo\ci\phi &\text{\ iff $\forall y(xR\ci y$ implies $\c M,y\models\phi$).}
\end{split}
 \end{equation}
 
 \begin{lemma}  \label{Mcirc}  
 \begin{enumerate}[\rm (1)]
 \item
For any formula $\phi$ and any $x$ in $W$, $\c M\ci,x\models\phi$ iff  $\c M,x\models\phi\ci$. 
\item
 $\c M\ci\models\phi$ iff  $\c M\models\phi\ci$.
 \item
  $\c F\ci\models\phi$ iff  $\c F\models\phi\ci$.
\end{enumerate}
\end{lemma} 

\begin{proof}
 \begin{enumerate}[\rm (1)]
 \item
By induction on formation of formulas. The crucial case is that of a formula $\di\phi$ under the hypothesis that the result holds for $\phi$. Then
$\c M\ci,x\models\di\phi$
iff there is some $y$ with $xR\ci y$ and $\c M\ci,y\models\phi$, which by the hypothesis holds iff there is some $y$ with $xR\ci y$ and $\c M,y\models\phi\ci$. The latter is equivalent by \eqref{split} to  $\c M,x\models \di\ci(\phi\ci)$, i.e.\ to
$\c M,x\models (\di\phi)\ci$.

\item
Follows from (1).
\item
Let   $\c F\ci\models\phi$. Then for any model $\c M$ on $\c F$ we have $\c M\ci\models\phi$, hence $\c M\models\phi\ci$ by (2). This shows $\c F\models\phi\ci$.
Conversely, assume $\c F\models\phi\ci$. Let  $\c M'=(\c F\ci,V)$ be any model on $\c F\ci$. Put $\c M=(\c F,V)$. Then $\c M\models \phi\ci$,   so   $\c M\ci \models \phi$  by (2). But $\c M\ci=\c M'$, so this shows that $\phi$ is verified by every model on $\c F\ci$.
 \qedhere
\end{enumerate} 
\end{proof}

Fine defines certain formulas $E,G,H$ and specifies $L$ to be the smallest normal extension of S4 to contain $G$ and $H$. He shows that
\begin{equation}  \label{Lvalid}
\text{any frame that validates $L$ also validates $\neg E$,}
\end{equation}
and then shows that $\neg E$ is not in $L$.

\begin{lemma}   \label{LnotE}
Let $L'$  be any  normal extension of $K4$ that contains $G\ci$ and $H\ci$.
Any frame that validates $L'$ also validates $\neg E\ci$.
\end{lemma}

\begin{proof}
Let $\c F\models L'$.  We show that $\c F\ci\models L$.
First, $\c F$ is transitive as it validates 4, so $\c F\ci$ is an S4-frame. 
Also $\c F \models G\ci \land H\ci$, hence 
$\c F\ci \models G\land H$  by Lemma \ref{Mcirc}(3).
Thus  the logic determined by  $\c F\ci$ includes $L$, i.e.\ $\c F\ci\models L$.

It follows by the result \eqref{Lvalid} that  $\c F\ci\models \neg E$. Hence by Lemma \ref{Mcirc}(3),   $\c F\models \neg E\ci$.
\end{proof}

The proof in  \cite{+Fine74c} that $\neg E$ is not in $L$ involves constructing a model that falsifies $\neg E$ but verifies all substitution instances of the axioms of $L$, hence verifies $L$. The frame of this model can be depicted as in Figure \ref{fineframe} (cf.~\cite[fig.~6.5]{ChagZak:ml}).
\begin{figure}
\begin{center}

\tikzset{->-/.style={decoration={
  markings,
  mark=at position #1 with {\arrow{>}}},postaction={decorate}}}

\begin{tikzpicture}[
scale=1.1, >=Stealth, 
V/.style={circle, fill=blue!0,  draw=blue!0, minimum size=22pt, inner sep=0pt}
]

\node (b0) at (0,4.5) [V]  {$b_0$} ;
\node (b1) at (2,4.5) [V]  {$b_1$} ;
\node (b2) at (4,4.5) [V]  {$b_2$} ;
\node (b3) at (6,4.5) [V]  {$b_3$} ;
\node (b4) at (8,4.5) [V]  {$b_4$} ;
\node (b5) at (10,4.5) [V]  {$b_5$} ;
\node (b6) at (11.2,4.5) []  {$\cdots$} ;

\node (c0) at (0,3) [V]  {$c_0$} ;
\node (c1) at (2,3) [V]  {$c_1$} ;
\node (c2) at (4,3) [V]  {$c_2$} ;
\node (c3) at (6,3) [V]  {$c_3$} ;
\node (c4) at (8,3) [V]  {$c_4$} ;
\node (c5) at (10,3) [V]  {$c_5$} ;
\node (c6) at (11.2,3) []  {$\cdots$} ;

\node (a0) at (4,1.5) [V]  {$a_0$} ;
\node (a1) at (6,1.5) [V]  {$a_1$} ;
\node (a2) at (8,1.5) [V]  {$a_2$} ;
\node (a3) at (10,1.5) [V]  {$a_3$} ;
\node (a4) at (11.2,1.5) []  {$\cdots$} ;

\node (d0) at (4,0) [V]  {$d_0$} ;
\node (d1) at (6,0) [V]  {$d_1$} ;
\node (d2) at (8,0) [V]  {$d_2$} ;
\node (d3) at (10,0) [V]  {$d_3$} ;
\node (d4) at (11.2,0) []  {$\cdots$} ;


\draw[->-=.7]  (b6) to  (b5);
\draw[->-=.6]  (b5) to  (b4);
\draw[->-=.6]  (b4) to  (b3);
\draw[->-=.6]  (b3) to  (b2);
\draw[->-=.6]  (b2) to (b1); 
\draw[->-=.6]  (b1) to  (b0);

\draw[->-=.7]  (c6) to  (c5);
\draw[->-=.6]  (c5) to  (c4);
\draw[->-=.6]  (c4) to  (c3);
\draw[->-=.6]  (c3) to  (c2);
\draw[->-=.6]  (c2) to (c1); 
\draw[->-=.6]  (c1) to  (c0);

\draw[->-=.7] (b2) to (c0); 
\draw[->-=.7] (c2) to (b0); 
\draw[->-=.7] (b3) to (c1); 
\draw[->-=.7] (c3) to (b1); 
\draw[->-=.7] (b4) to (c2); 
\draw[->-=.7] (c4) to (b2); 
\draw[->-=.7] (b5) to (c3); 
\draw[->-=.7] (c5) to (b3);

\draw[->-=.65]  (d0) -- (a0);
\draw[->-=.65]  (d1) to (a1);
\draw[->-=.65]  (d2) to (a2);
\draw[->-=.65]  (d3) to (a3);
\draw[->-=.65]  (d0) to  (d1);
\draw[->-=.65]  (d1) to  (d2);
\draw[->-=.65]  (d2) to (d3); 
\draw[->-=.75]  (d3) to  (d4);

\draw[->-=.32]  (a0) to [in=295,out=135]  (b1);
\draw[->-=.55]  (a0) to [in=310,out=155]  (c1);
\draw[->-=.32]  (a1) to [in=295,out=135] (b2);
\draw[->-=.55]  (a1) to  [in=310,out=155] (c2);
\draw[->-=.32]  (a2) to [in=295,out=135] (b3);
\draw[->-=.55]  (a2) to  [in=310,out=155] (c3);
\draw[->-=.32]  (a3) to [in=295,out=135] (b4);
\draw[->-=.55]  (a3) to  [in=310,out=155] (c4);

\draw[->-=.52]  (d4) to  [in=0,out=0] (b6);
\draw[->-=.7]  (d4) to  [in=0,out=5] (c6);

\end{tikzpicture}

\caption{An irreflexive transitive frame validating $U_2$ but not $U_1$}
\label{fineframe}
\end{center}
\end{figure}
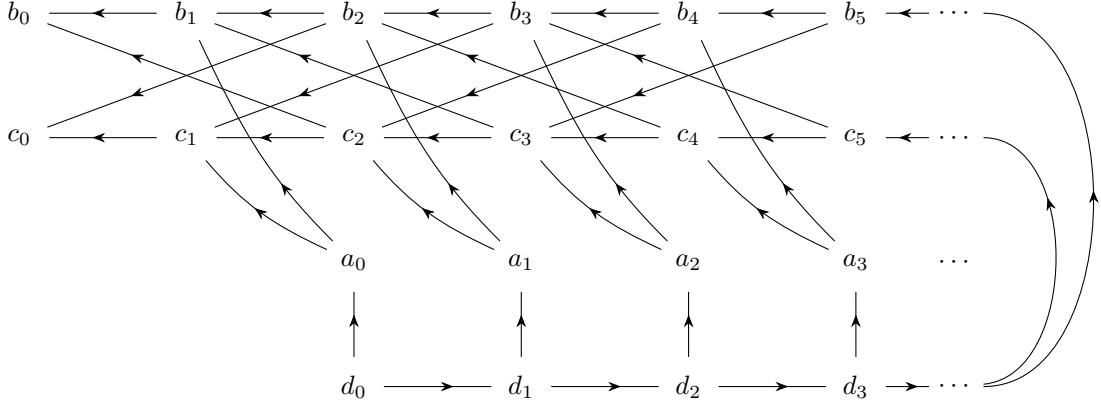
The labelling gives a four-colouring, partitioning the nodes into $a$-points, $b$-points etc. The frame  relation is depicted by the arrows. Note that the $a$-points form an infinite antichain, and the $d$-points form an infinite ascending $R$-chain, the only one in the frame other than subsequences of the $d$-points. The frame is point-generated by $d_0$.

In \cite{+Fine74c} the relation was taken to be the reflexive transitive relation generated by the displayed arrows. It is a partial ordering. Here we do not need the relation to be reflexive. For the rest of this section, let $\c F=(W,R)$ be the irreflexive transitive version of this frame.  Then Fine's frame is $\c F\ci$. The relation $R$, which is asymmetric, can be specified by listing all the future sets:

\bigskip
$R(b_0)=R(c_0)=\emptyset$,

$R(b_1)=\{b_0\}$,\enspace  $R(c_1)=\{c_0\}$,

$R(b_{m+2})=\{b_0,\dots, b_{m+1}, c_0,\dots,c_m\}$,

$R(c_{m+2})=\{b_0,\dots, b_{m}, c_0,\dots,c_{m+1}\}$,

$R(a_{m})=\{b_0,\dots, b_{m+1}, c_0,\dots,c_{m+1}\}$,

$R(d_{m})=\{d_n:m<n\}\cup\bigcup\{R\ci(a_n):m\leq n\}$.

\bigskip\noindent
$\c F$ fails to validate $U_1$ in infinitely many ways. For instance, it contains the pair $(b_1,c_1)$ which is achronal as their futures $\{b_0\}$ and $\{c_0\}$ are $\subseteq$-incomparable. Similarly, every pair $(b_{m+1},c_{m+1})$ is achronal.

On the other hand $\c F$ validates $U_2$, since it has no triple of points that is achronal. To see why, note that the sequence
$(R(b_m):m<\omega)$ of future sets of $b$-points increases monotonically under $\subseteq$ as $m$ increases, so any two such sets are comparable. Likewise the $R(c_m)$'s are increasing, and so are the $R(a_m)$'s, while the $R(d_m)$'s are decreasing. So any two points of the same colour have comparable futures, hence cannot both belong to an achronal triple.

Also, a $d$-point cannot belong to any achronal triple, as the future of any $a$, $b$ or $c$ point consists entirely of $b$ and $c$ points, while the future of any $d$ point includes all of the $b$ and $c$ points. Hence the future of any $d$ point is comparable to that of any other point. 

That leaves the only possibility for an achronal triple being that it consists of one point of each of the colours $a,b,c$.
However $R(a_m)\supseteq R(b_{m+2})\supseteq R(b_{m+1})\supseteq\cdots\supseteq R(b_{0})$, while  
$R(a_m)\subseteq R(b_{m+k})$ for all $k\geq 3$. Hence the future of any $a$ point is comparable to that of any $b$ point (and similarly of any $c$ point). So this possibility is ruled out, and no achronal triple exists.

In  \cite{+Fine74c}, a model $(\c F\ci,V)$ is defined that satisfies $E$ at $d_0$. Let $\c M=(\c F,V)$ be the model on $\c F$ with the same valuation, so that $(\c F\ci,V)$ is $\c M\ci$, and $\c M\ci\not\models \neg E$. To prove that $\c M\ci$ verifies $L$ it was shown that $\c F\ci\models G$, and that $\c M\ci$ \emph{strongly verifies} $H$, i.e.\ it verifies every substitution instance of $H$. Thus all axioms of $L$ are strongly verified by 
 $\c M\ci$. But the rules of inference preserve strong verification, so  $\c M\ci$ (strongly) verifies $L$.

By Lemma \ref{Mcirc}, from  $\c M\ci\not\models \neg E$ we get
$\c M\not\models \neg E\ci$, and from $\c F\ci\models G$ we get $\c F\models G\ci$. We also have $\c F\models U_2$.
\textit{If we can show that $\c M$ strongly verifies $H\ci$}, then  putting $L' = K4U_2G\ci H\ci$ we get $\c M\models L'$, so 
$\neg E\ci\notin L'$. Hence by Lemma \ref{LnotE}, $L'$ is a Kripke incomplete extension of $K4U_2$.

So far we have not needed to know what $E,G,H$ are, given the results about them in \cite{+Fine74c}. But
Lemma  \ref{Mcirc} does not appear to give access to a proof that $\c M$ strongly verifies $H\ci$, since a substitution instance of $H\ci$ need not be of the form $\phi\ci$ where $\phi$ is  a substitution instance of $H$. So we must revisit the proof in \cite{+Fine74c} that $\c M\ci$ strongly verifies $H$ and adapt it to show that $\c M$ strongly verifies $H\ci$.
The formula $H$ is
$$
\neg(s\land\bo(s\to\di(\neg s\land t\land\di(\neg s\land\neg t\land\di s)))),
$$
where $s$ and $t$ are atoms, so  $H\ci$ is
$$
\neg(s\land\bo\ci(s\to\di\ci(\neg s\land t\land\di\ci(\neg s\land\neg t\land\di\ci s)))).
$$
To show that this is strongly verified by our model $\c M$, suppose otherwise for the sake of contradiction. Then there are formulas $\phi$ and $\psi$ and some $w\in W$ with 
$$
\c M,w \models  \phi\land\bo\ci(\phi\to\di\ci(\neg \phi\land \psi\land\di\ci(\neg \phi\land\neg \psi\land\di\ci \phi))).
$$
Hence in $\c M$ there are points $x,y,z$ with $wR\ci xR\ci yR\ci z$, such that $w\models\phi$, $x\models \neg\phi\land\psi$,
$y\models \neg\phi\land\neg\psi$ and $z\models\phi$. Then $w\ne x\ne y\ne z$, so $wRxRyR z$. Since $z\models \phi$, we can use $H\ci$ again to repeat the argument to obtain further points $zRx'Ry'Rz'\models\phi$ and so on ad infinitum,  generating an infinite ascending $R$-chain from $w$. This means we must have $w=d_m$ for some $m$, with the chain consisting of $d$-points, and
the formulas  $\phi$, $\neg\phi\land\psi$, $\neg \phi\land\neg \psi$ being true cofinally along the chain.

The rest of the argument to a contradiction is as in \cite{+Fine74c}.
 We repeat the details. Let $\bo\chi_1,\dots\bo\chi_k$ be all the formulas from which $\phi$ and $\psi$ can be constructed by Boolean connectives.
Since the $R(d_i)$'s are decreasing, if $d_i\models\bo\chi_l$, then  $d_j\models\bo\chi_l$ for all $j>i$.
Hence there exists some $n\geq m$ such that the truth values of $\bo\chi_1,\dots,\bo\chi_k$ are fixed from $d_n$ on, i.e.\ 
$d_i\models \bo\chi_l$ iff $d_n\models \bo\chi_l$, for all $i\geq n$ and $l\leq k$.

We have not yet had to say anything about the valuation $V$. In fact there are two atoms $p_0$ and $p_1$ with 
$V(p_0)=\{d_{i}:i\ \text{is even}\}$ and $V(p_1)=\{d_{i}:i \ \text{is odd}\}$, while no other atoms are true at any  $d$-point. So if $i,j\geq n$ have the same parity (even/odd), then $d_i$ and $d_j$ agree on the truth values of all atoms and the formulas  $\bo\chi_1,\dots\bo\chi_k$, so they agree on $\phi$ and $\psi$ and all their Boolean combinations.

But there exist three $d$-points beyond $d_n$ at which the formulas $\phi$, $\neg\phi\land\psi$, and $\neg \phi\land\neg \psi$ successively are true. Then at least two of these points must have the same parity of index, hence must agree on the truth-values of the formulas, which is impossible. That is the contradiction confirming that $\c M$ strongly verifies $H\ci$, finishing the proof that  $K4U_2G\ci H\ci$ is Kripke incomplete.

\medskip

The following is not settled by this argument:

\begin{problem}\label{prob:K4U1 ext}
Are all normal monomodal logics extending $K4U_1$
Kripke complete?
\end{problem}

\paragraph{Conclusion}
On the one hand then, 
compared with Fine's $K4I_n$, 
the logics $KU_n$ are rather weak (or small).
They do not contain the transitivity axiom 4,
there are continuum-many normal modal logics between $K4U_n$ and $K4I_n$,
and some normal extensions of $KTU_1 $ and of $K4U_2$
are Kripke incomplete.
This is good news inasmuch as 
it indicates that theorem~\ref{thm:can} below
covers a wide range of logics.
On the other hand, 
the frame conditions imposed by the $U_n$ are substantial,
as it is easy to find frames
such as the lawn-rake frame $\c F_\omega$
that do not validate any $KU_n$.

\section{Canonicity for logics above $KU_n^\M$}\label{sec:main canon}

We are going to prove (in theorem~\ref{thm:can}) that all $\omega$-canonical 
normal modal $\M$-logics of finite \ach\ width (see definition~\ref{def:Un}) are totally canonical.
By the inclusions shown in figure~\ref{fig:logics}, 
this holds for normal extensions of $K5_2$, $K5$, and $K4I_n$ as well.
In terms of varieties, we will prove that 
every variety that validates some $U_n$ is \nice.
The proof extends the one in \S\ref{ss:K52 canon}.

\subsection{Infinite \ach\ sets in subsets of frames}

We start by extending lemma~\ref{lem:valid52} to this setting.
For any $n$, validity of the axioms $U_n^\M$ in a frame $\c F$
guarantees that no set $R^\c F_\di(x)$ (for $x\in\c F$ and $\di\in\M$)
has an infinite \ach\ subset.
But does this extend to sets of the form
$R_s(x)=\{y\in\c F:\c F\models R_sxy\}$ for arbitrary $s\in{}^{<\omega}\M$?
(See definition~\ref{def:etc} for $R_s$.)

\begin{example}\rm

There do exist  frames $\c F$
(monomodal ones, with $\M=\{\di\}$)
in which
 no $R^\c F_\di(x)$ (for any $x\in\c F$) has
 an infinite \ach\ subset
but 
$R_s(x)=\{y\in\c F:\c F\models R_sxy\}$ does, for some $x\in\c F$,
where $s=(\emptyset\hat\ \di)\hat\ \di$ is the sequence of two $\di$s.

Let $\c F=(\{a\}\cup(\omega\times2),R)$,
where $a\notin\omega\times2$ and
$R=\{(a,(i,0)),((i,0),(j,1)), ((i,1),(i,1)):j\leq i<\omega\}$.
Then $R(a)=\omega\times\{0\}$, which does not contain
even a two-element \ach\ set,
since 
$\{R((i,0)):i<\omega\}$ forms a chain under inclusion.
And $R(x)$ is finite for all  $x\in\c F\setminus\{a\}$, so certainly has no infinite \ach\ subsets.
But $R_s(a)=\omega\times\{1\}$ is infinite and \ach,
since 
if $i<j<\omega$ then
$R((i,1))=\{(i,1)\}$
and $R((j,1))=\{(j,1)\}$ are $\subseteq$-incomparable.

\end{example}

Inspired by lemma~\ref{lem:valid52}, we show in the next lemma that
this cannot happen for frames validating some $U_n^\M$.
It will be used in proposition~\ref{prop:cof}.
A more effective version could be proved using Ramsey numbers.

\begin{lemma}\label{lem:ramsey}
Let $0<n<\omega$ and let $\c F$ be a frame validating $U^\M_n$.
Let
$s\in{}^{<\omega}\M$
and  $x\in\c F$,
and write $R_s(x)$ for the set
$\{y\in\c F:\c F\models R_sxy\}$.
Then every \ach\ subset of $R_s(x)$ is finite.
\end{lemma}

\begin{proof}
In the proof, we write $R^\c F_\di(y)$ as just $R_\di(y)$, for $y\in\c F$ and $\di\in\M$.
The proof is by induction on the finite ordinal $\dom(s)$.
For $s=\emptyset$, the lemma holds trivially, since $R_\emptyset(x)=\{x\}$ is itself finite.
Assume inductively that the lemma holds for $s\in{}^{<\omega}\M$ (for all $x\in\c F$),
and let $\di\in\M$ be given. We prove 
that every \ach\ subset of  $R_{s\hat\ \di}(x)$ is finite.

By lemma~\ref{lem:pif},
it is enough to show the following.
Let $y_0,y_1,\ldots\in R_{s\hat\ \di}(x)$,
and let $\bdi\in\M$ be given.
Then $R_\bdi(y_i)\subseteq R_\bdi(y_j)$
for some distinct $i,j<\omega$.

Choose $z_0,z_1,\ldots\in R_{s}(x)$
such that $y_i\in R_\di(z_i)$ for each $i<\omega$.
By Ramsey's theorem \cite{Ram30}, there is infinite 
$K=\{k_0,k_1,\ldots\}\subseteq\omega$,
where $k_0<k_1<\cdots$,
such that
\begin{enumerate}
\item $R_\di(z_k)\subseteq R_\di(z_l)$ for each $k<l$ in $K$, or

\item $R_\di(z_k)\supset R_\di(z_l)$ for each $k<l$  in $K$, or

\item $R_\di(z_k)\not\subseteq R_\di(z_l)\not\subseteq R_\di(z_k)$ for each $k<l$  in $K$.
\end{enumerate}
The last option is impossible, since
if it held, $(z_{k_i}:i<\omega)$ would be an infinite
achronal sequence in $R_s(x)$, 
which by lemma~\ref{lem:pif} would contradict the inductive hypothesis for $s$.
If option~1 holds, then
$y_{k_0},\ldots,y_{k_n}\in R_\di(z_{k_n})$.
If option~2 holds, then
$y_{k_0},\ldots,y_{k_n}\in R_\di(z_{k_0})$.
Either way,
$y_{k_0},\ldots,y_{k_n}\in R_\di(t)$ for some $t\in\c F$.
Since $U_n^\M$ is valid in $\c F$ at  $t$,
by~\eqref{e:fo cor} in \S\ref{sss:corresps}
 there must be distinct $i,j\leq n$
with  $R_\bdi(y_{k_i})\subseteq R_\bdi(y_{k_j})$, as required.
This completes the induction and proves the lemma.
\end{proof}

\subsection{Underlying}

We now return to the variety $V$ of \S\ref{ss:var V}.
From now on, we assume that \emph{the logic of $V$ is 
of finite \ach\ width --- that is, above some $KU_n^\M$.}
Explicitly, for some integer $n\geq1$, the translations 
$\tau_\varphi$ (given in \S\ref{ss:algs})
of
the axioms  $\varphi\in U_n^\M$ to $\M$-BAO terms 
have value 1 in every algebra in $V$ under all assignments to their variables.
In terms of \S\ref{ss:algs},
we could just say that $V$ validates some $KU^\M_n$.

We will show in theorem~\ref{thm:can} that $V$ is \nice.
We make free use of material in \S\S\ref{sec:defs}--\ref{sec:setup}.

\begin{definition}
For $\di\in\M$,
we let $\o R_\di xy$
stand for the $\c L$-formula
$\forall z(R_\di yz\to R_\di xz)$.
\end{definition}
This recalls  $\o R$ in \S\ref{ss:S4In=S4Un}.
For $t,u\in\c B_+$,
we have $\gc B\models \o R_\di tu$
iff $R_\di^{\c B_+}(u)\subseteq R_\di^{\c B_+}(t)$:
the $\di$-future of $u$ in the frame $\c B_+$ is contained in that of $t$.
Then a subset of $\c B_+$ is $\di$-\ach\
iff it is an antichain with respect to $\o R_\di$.

\begin{definition}
\begin{enumerate}

\item Let $S,T\subseteq \c B_+$. We say that
\emph{$T$ underlies $S$}
if for each $s\in S$
and $\di\in \M$, there is $t\in T$ with
$\gc B\models\o R_\di st$.
This is a kind of density property.

\item As is standard, a subset $D\subseteq\gc B$
 is said to be \emph{$\c L(\gc A)$-definable}
if  $D=\{d\in\gc B:\gc B\models\gamma(d)\}$
for some $\c L(\gc A)$-formula $\gamma(x)$.
\end{enumerate}
\end{definition}

The following is the key technical proposition
in the proof.  It will be used in lemma~\ref{lem:truth}.
The embedding $f:\c A_+\to\c B_+$ is as in definition~\ref{def:f}.

\begin{proposition}\label{prop:cof}
Let $\c F$ be an inner subframe of $\c A_+$ generated by an element of $\wh W$,
and let $D$
be an $\c L(\gc A)$-definable subset of $\c B_+$.
Then $W\cap D$ underlies
$f(\c F)\cap D$.
\end{proposition}

\begin{proof}
Choose $w_0\in W$ such that $\c F=\c A_+(\wh{w_0})$. 
Let
$D=\{d\in\c B_+:\gc B\models\gamma(d)\}\subseteq\c B_+$,
where $\gamma(x)$ is  some arbitrary $\c L(\gc A)$-formula with $x$ of point sort.
Let $\mu\in\c F$ with  $f(\mu)\in D$,
and let $\di\in\M$.
We show that  there is $w\in W\cap D$ with
$\gc B\models\o R_\di(f(\mu),w)$.

Clearly, if $f(\mu)\in W$ then we can take $w=f(\mu)$.
So we can suppose that
\begin{equation}\label{e:loc mu}
f(\mu)\in D\setminus W.
\end{equation}

There is no loss of generality if we 
replace $D$ by a `nicer' $\c L(\gc A)$-definable subset $E$ of $D$
still containing $f(\mu)$.
We will take advantage of this.
First,
by definition of $\c F$ (definition~\ref{def:etc}),
we can choose $s\in{}^{<\omega}\M$
with
\begin{equation}\label{e:loc mu 2}
\c A_+\models R_s (\wh{w_0},\mu).
\end{equation}
Moving from $\c A_+$ to $\gc B$,
let
\begin{equation}\label{e:D}
D'=\{d\in D: \gc B\models R_s  {w_0}d\}.
\end{equation} 
Now let $K\subseteq W\cap D'$ be maximal such that
$K\cup\{f(\mu)\}$ is $\di$-\ach\ in the frame $\c B_+$.
%
Such a $K$ exists by Zorn's lemma; it may be empty.
Since $\c B\in V$ and the axioms in $U^\M_n$ are Sahlqvist,
it follows that $\c B_+$ validates $U^\M_n$,
so by~\eqref{e:D} and lemma~\ref{lem:ramsey},
there are no infinite \ach\ subsets of $D'$.
Hence, $K$ is finite.
So if we let
\begin{equation}\label{e:E}
E=\{x\in D'\setminus K: K\cup\{x\}
\mbox{ is $\di$-\ach\ in }\c B_+
\},
\end{equation}
then $E$ is $\c L(\gc A)$-definable,
by the $\c L(\gc A)$-formula
\begin{equation*}
\varepsilon(x)=\gamma(x)\wedge R_s \u{w_0}x\wedge
\neg\bigvee_{k\in K}\o R_\di x\u k\vee\o R_\di \u kx.
\end{equation*}
(We do not need to add a disjunct $x=\u k$ at the end here.)

By~\eqref{e:E} and~\eqref{e:D},
 $E\subseteq D'\subseteq D$.
Also, $\gc B\models R_s(w_0,f(\mu))$
by \eqref{e:loc mu 2} and  corollary~\ref{cor:f emb};
so by~\eqref{e:loc mu} and~\eqref{e:D}, $f(\mu)\in D'$.
By~\eqref{e:loc mu} again,
$f(\mu)\notin W\supseteq K$.
So by~\eqref{e:E}
and the definition of~$K$,
\begin{equation}\label{e:fuinE}
f(\mu)\in E.
\end{equation}

\medskip

We now proceed to find $w\in W\cap E$ with
$\gc B\models\o R_\di(f(\mu),w)$ 
--- that is, $R^\gc B_\di(w)\subseteq R^\gc B_\di(f(\mu))$.
Assume for contradiction that $(\dag)$
there is no such $w$.


\Claim1
$\gc B\models\o R_\di(w,f(\mu))$
for every $w\in W\cap E$.
 
\pfclaim
Let $w\in W\cap E$.
Then $w\notin K$ by  \eqref{e:E},
so $K\subset K\cup\{w\}\subseteq W\cap D'$.
The maximality of $K$ now implies that
$K\cup\{w,f(\mu)\}$ is not $\di$-\ach\ in $\c B_+$.
That is, it is not an $\o R_\di$-antichain in $\gc B$.
Since $K\cup\{w\}$ and $K\cup\{f(\mu)\}$  are clearly $\di$-\ach, 
so are such antichains,
we must have
$\gc B\models\o R_\di(w,f(\mu))$
or
 $\gc B\models\o R_\di(f(\mu),w)$.
Our assumption $(\dag)$ rules out the second of these,
so indeed,
$\gc B\models\o R_\di(w,f(\mu))$.
Since $w$ was arbitrary, this proves the claim.

\medskip

We now extend this to the whole of $E$.

\penalty-100

\Claim2
$\gc B\models\o R_\di(e,f(\mu))$ for every $e\in  E$.

\pfclaim
First, it is an exercise in canonical frames to show that for every $t,u\in\c B_+$,
\begin{equation}\label{e:di R}
\gc B\models \o R_\di tu
\leftrightarrow \forall x(\di x\inn u\to\di x\inn t).
\end{equation}

Now fix any $w\in W\cap E$.
Claim~1 says that
$\gc B\models\o R_\di(w,f(\mu))$.
So by~\eqref{e:di R},
we have $\gc B\models\forall x(\di x\inn f(\mu)\to \di x\inn w)$.
By definition of  $f$ (see~\eqref{e:def f} in \S\ref{ss:def f}), 
this says that
\begin{equation*}\label{e:w thing}
\gc B'\models\forall x(\u{\mu}(\di x)\to \di x\inn\u w).
\end{equation*}
This  is a universally-quantified simple formula,
so by corollary~\ref{cor:A case},
$\gc A'\models\forall x(\u{\mu}(\di x)\to \di x\inn\u w)$ as well.
Now this holds for all $w\in W\cap E=\gc A\cap E$.
And since $\gc A\preceq\gc B$,
for $w\in W$
we have $w\in E$ iff $\gc B\models\varepsilon(w)$, iff $\gc A\models\varepsilon(w)$.
So all in all,
\begin{equation*}\label{e:xi}
\gc A'\models\forall xy(\underbrace{\varepsilon(y)\wedge\u{\mu}(\di x)\to \di x\inn y}_\theta).
\end{equation*}
Here, $\theta$  is a simple $\c L'$-formula.
So by corollary~\ref{cor:A case} again,
$\gc B'\models\forall xy(\varepsilon(y)\wedge\u{\mu}(\di x)\to \di x\inn y)$
---
this is a key step in the proof.
By~\eqref{e:def f} again, 
$\gc B\models\forall xy(\varepsilon(y)\wedge\di x\inn f(\mu)\to\di x\inn y)$,
so by~\eqref{e:di R},
\begin{equation}\label{e:nearly}
\gc B\models\forall y(\varepsilon(y)\to\o R_\di(y,f(\mu))).
\end{equation}

Finally, take any $e\in E$.
Then $\gc B\models\varepsilon(e)$, so~\eqref{e:nearly} yields
$\gc B\models\o R_\di(e,f(\mu))$,
proving the claim.

\bigskip

We said in~\eqref{e:fuinE} that $f(\mu)\in E$,
and combined with
claim~2, this gives
\begin{equation}\label{e:epsilon}
\gc B\models\varepsilon(f(\mu))\wedge\forall y(\varepsilon(y)
\to\o R_\di(y,f(\mu))).
\end{equation}
Thus, $f(\mu)$ witnesses
 $\gc B\models\exists x[\varepsilon(x)\wedge
 \forall y(\varepsilon(y)\to \o R_\di yx)]$.
By the Tarski--Vaught criterion (\S\ref{ss:A<B}),
there is $w\in W$
with 
\begin{equation}\label{e:ten}
\gc B\models\varepsilon(w)\wedge
\forall y(\varepsilon(y)\to\o R_\di yw).
\end{equation}
Now, taking $y=f(\mu)$ in~\eqref{e:ten},
and recalling from~\eqref{e:epsilon} that $\gc B\models\varepsilon(f(\mu))$,
we get
$\gc B\models\o R_\di(f(\mu),w)$.
Since \eqref{e:ten} implies $w\in W\cap E$, this contradicts
our assumption $(\dag)$ that there is no such $w$,
and proves the proposition.
\end{proof}

\subsection{Canonicity}\label{ss:canon}

We now prove that the Truth Lemma
of definition~\ref{def:truth lem}  holds in $V$.
Define $g$, $\c F$, $\c M$ as in \S\ref{ss:strategy}.
Recall the standard translation $\psi^x$
from \S\ref{ss:strategy}.
Recalling~\eqref{e:stdtr} from that section, for each $\M$-formula $\psi$
and $t\in\c B_+$ we have
\begin{equation}\label{e:stdtr 2}
(\c B_+,g),t\models\psi \iff \gc B\models\psi^x(t).
\end{equation}

\begin{lemma}[Truth Lemma]\label{lem:truth}
For each $m\in\c M$ and modal $\M$-formula $\psi$,
we have $\c M,m\models\psi$ iff $(\c B_+,g),m\models\psi$.

\end{lemma}

\begin{proof}
By induction on $\psi$.
The main case is of course $\di\psi$
(where $\di\in \M$).
Assume the lemma inductively for $\psi$.
It is easy to check
(as in the proof of proposition~\ref{prop:tl for 52})
that if $\c M,m\models\di\psi$ then $(\c B_+,g),m\models\di\psi$ as well.
Conversely, suppose that $(\c B_+,g),m\models\di\psi$.
We will show that $\c M,m\models\di\psi$.

Let $D\subseteq\c B_+$ be the set defined in $\gc B $ by $(\di\psi)^x$.
By~\eqref{e:stdtr 2}, $m\in\c M\cap D$.
By proposition~\ref{prop:cof}, $W\cap D$ underlies $\c M\cap D$,
so
there is $w\in W\cap D$ with 
$\gc B\models\o R_\di mw$ --- that is,
\begin{equation}\label{e:cof wit}
R_\di^\gc B(w)\subseteq R_\di^\gc B(m).
\end{equation}

Since $w\in D$, we have $\gc B\models(\di\psi)^x(w)$,
which is to say that $\gc B\models\exists y(R_\di\u wy\wedge\psi^y)$.
So by the Tarski--Vaught criterion,
there is $w'\in W$
with
\begin{equation}\label{e:psi at w'}
\gc B \models R_\di ww'\wedge\psi^y(w').
\end{equation}
So by~\eqref{e:stdtr 2},
\begin{equation}\label{e:psi at w' 2}
(\c B_+,g),w'\models\psi.
\end{equation}
By~\eqref{e:psi at w'} and \eqref{e:cof wit}, 
$w'\in R_\di^\gc B(w)\subseteq R_\di^\gc B(m)$, so
\begin{equation}\label{e:underlies use}
\gc B\models R_\di mw'.
\end{equation}

It follows that $w'\in\c M$.
To see this, choose $\mu\in\c F$ with $m=f(\mu)\strut$.
By corollary~\ref{cor:f emb}, 
$f:\c A_+\to\c B_+$ is a frame embedding and $w'=f(\wh{w'})$.
So by~\eqref{e:underlies use},
$\c A_+\models R_\di \mu\wh{w'}\strut$.
By definition, $\c F$ is $R^{\c A_+}_\di$-closed in $\c A_+$, and since $\mu\in\c F$,
we get $\wh{w'}\in\c F$ and so $w'\in \c M$.

This means that
we can apply the inductive hypothesis to \eqref{e:psi at w' 2}, giving
$\c M,w'\models\psi$. 
Also, $\c M$ is a submodel of $(\c B_+,g)$,
so \eqref{e:underlies use} gives
$\c M\models R_\di mw'$.
We conclude that $\c M,m\models\di\psi$,
 completing the induction.
%
\end{proof}

\begin{theorem}\label{thm:can}
Assume that the logic of $V$ is of finite \ach\ width
--- above  $KU_n^\M$, for some $0<n<\omega$.
Then $V$ is \nice.
\end{theorem}

\begin{proof}
Immediate from lemma~\ref{lem:truth} and theorem~\ref{thm:truthlemma}.
\end{proof}

\subsection{Final remarks}\label{ss:final rmks}

Theorem~\ref{thm:can} expresses a dichotomy:
every variety of $\M$-BAOs (or normal modal logic) of finite \ach\ width
is either totally canonical or not even $\omega$-canonical.
It would be rather unsatisfactory
if all examples fell on the same side of the line.
But while $KU^\M_1,KU^\M_2,\ldots$ 
themselves are canonical,
 not all monomodal logics extending even
$K4I_1$ are canonical.
An example is S4.3Grz,
the smallest  monomodal logic
containing $p\to\di p$,
the axioms 4 and $I_1$ from~\eqref{e:In corr},
and 
$\mathrm{Grz}=\bo(\bo(p\to\bo p)\to p)\to p$
--- see~\cite[p.38]{+Fine74}.

Temporal logics and other logics with `converse diamonds'
are covered by theorem~\ref{thm:can}.
To get that $\di,\bdi\in\M$ are 
mutually converse in all canonical frames of algebras in $V$, we simply require that
$V$ validate the Sahlqvist axioms
$p\to\bbo\di p$ and $p\to\bo\bdi p$, regarded as $\M$-equations as usual.
For $\M=\{\di,\bdi\}$, any normal modal $\M$-logic containing these axioms,
axiom~4 for $\di$
from~\eqref{e:In corr}, and linearity axioms $I_1$ for $\di,\bdi$, also contains
 $U_1^\M$ and is a linear temporal logic covered by theorem~\ref{thm:can}.

From \S\ref{ss:M ctble} onwards, we took $\M$ to be countable.
This assumption was made only 
for presentational simplicity.
It had little to no effect on \S\S\ref{sec: K5}--\ref{sec:kr comp},
where $\M$ was usually a singleton $\{\di\}$ anyway.
But  we  could allow $\M$ to have uncountable cardinality $\kappa$ throughout the paper.
The only change required to the proofs is
that the elementary sub\str\ $\gc A\preceq\gc B$ 
in \S\ref{ss:A<B}
be taken of cardinality~$\leq\kappa$,
rather than countable.
The meaning of theorem~\ref{thm:can} becomes:
if the logic of $V$ is of finite \ach\ width
and $V$ is $\kappa$-canonical then it is totally canonical.
However, the theorem is formulated using \nice ness and is true as stated for all $\M$ of
any cardinality.

\section{Conclusion}\label{sec:conc}

We have shown that every variety of $\M$-BAOs validating $K5_2$,
or $KU_n$ for some finite $n\geq1$, is \nice.
Though we have made some progress on special cases,
Fine's problem from~\cite{Fine75} in its full generality remains open
and we still do not know whether every variety of $\M$-BAOs is \nice.
We also have no answer to 
problems~\ref{prob:f elem} and~\ref{prob:K4U1 ext},
as well as the following:


\begin{problem}
For finite $n\geq1$,
are all  normal monomodal extensions of 
Xu's logic $K4\mathsf{Wid}_n^*$
from \cite{Xu21}
\nice?
\end{problem}

\begin{problem}
Is every \emph{canonical} normal multimodal logic of 
finite \ach\ width the logic of an elementary class of frames?
\end{problem}

We have taken each $\di\in\M$ to be a \emph{unary} modal operator.
It may be worth investigating whether the methods of this paper
extend to polyadic  modal signatures.



\begin{thebibliography}{10}

\bibitem{+vb78}
J.~van Benthem, \emph{Two simple incomplete modal logics}, Theoria \textbf{44}
  (1978), 25--37.

\bibitem{BRV:ml}
P.~Blackburn, M.~de~Rijke, and Y.~Venema, \emph{Modal logic}, Tracts in
  Theoretical Computer Science, Cambridge University Press, Cambridge, UK,
  2001.

\bibitem{+blok80}
W.~Blok, \emph{The lattice of modal logics: an algebraic investigation}, J.
  Symbolic Logic \textbf{45} (1980), 221--236.

\bibitem{ChagZak:ml}
A.~Chagrov and M.~Zakharyaschev, \emph{Modal logic}, Oxford Logic Guides,
  vol.~35, Clarendon Press, Oxford, 1997.

\bibitem{ChK90}
C.~C. Chang and H.~J. Keisler, \emph{Model theory}, 3rd ed., North-Holland,
  Amsterdam, 1990.

\bibitem{+Fine74c}
K.~Fine, \emph{An incomplete logic containing {S4}}, Theoria \textbf{40}
  (1974), 23--29.

\bibitem{+Fine74}
\bysame, \emph{Logics containing {K4}. {P}art {I}}, J. Symbolic Logic
  \textbf{39} (1974), 31--42.

\bibitem{Fine75}
\bysame, \emph{Some connections between elementary and modal logic}, Proc.\ 3rd
  {S}candinavian logic symposium, {U}ppsala, 1973 (S.~Kanger, ed.), North
  Holland, Amsterdam, 1975, pp.~15--31.

\bibitem{fine:logi85}
\bysame, \emph{Logics containing {K}4. {P}art {II}}, J. Symbolic Logic
  \textbf{50} (1985), no.~3, 619--651.

\bibitem{Gol87}
R.~Goldblatt, \emph{Logics of time and computation}, CSLI Lecture Notes,
  vol.~7, The Chicago University Press, Chicago, 1987.

\bibitem{gold:vari89}
\bysame, \emph{Varieties of complex algebras}, Annals of Pure and Applied Logic
  \textbf{44} (1989), 173--242.

\bibitem{Gol95:canon}
\bysame, \emph{Elementary generation and canonicity for varieties of boolean
  algebras with operators}, Algebra Universalis \textbf{34} (1995), 551--607.

\bibitem{Hodg93}
W.~Hodges, \emph{Model theory}, Cambridge University Press, 1993.

\bibitem{JT51}
B.~J{\'{o}}nsson and A.~Tarski, \emph{Boolean algebras with operators {I}},
  American Journal of Mathematics \textbf{73} (1951), 891--939.

\bibitem{nagl:exte85}
Michael~C. Nagle and S.~K. Thomason, \emph{The extensions of the modal logic
  {K}5}, J. Symbolic Logic \textbf{50} (1985), no.~1, 102--109.

\bibitem{Ram30}
F.~P. Ramsey, \emph{On a problem of formal logic}, Proc. London Math. Soc.
  \textbf{30} (1930), 264--286.

\bibitem{sege:deci68a}
Krister Segerberg, \emph{Decidability of four modal logics}, Theoria
  \textbf{34} (1968), 21--25.

\bibitem{Surendonk96}
T.~J. Surendonk, \emph{A non-standard injection between canonical frames},
  Logic J. IGPL \textbf{4} (1996), 273--282.

\bibitem{Surendonk97}
\bysame, \emph{On isomorphisms between canonical frames}, Advances in Modal
  Logic (AiML'96) (M.~Kracht, M.~de~Rijke, H.~Wansing, and M.~Zakharyaschev,
  eds.), vol.~1, CSLI Publications, 1997, pp.~249--268.

\bibitem{Surendonk98}
\bysame, \emph{Canonicity for intensional logics}, Ph.D. thesis, Australian
  National University, 1998.

\bibitem{Tho72}
S.~K. Thomason, \emph{Semantic analysis of tense logic}, J. Symbolic Logic
  \textbf{37} (1972), 150--158.

\bibitem{wolt:prop96}
Frank Wolter, \emph{Properties of tense logics}, Mathematical Logic Quarterly
  \textbf{42} (1996), 481--500.

\bibitem{wolt:stru97}
\bysame, \emph{The structure of lattices of subframe logics}, Annals of Pure
  and Applied Logic \textbf{86} (1997), 47--100.

\bibitem{Xu21}
M.~Xu, \emph{Transitive logics of finite width with respect to
  proper-successor-equivalence}, Studia Logica \textbf{109} (2021), 1177--1200.

\bibitem{zakh:cano96}
Michael Zakharyaschev, \emph{Canonical formulas for {K}4. {P}art {II}: Cofinal
  subframe logics}, J. Symbolic Logic \textbf{61} (1996), no.~2, 421--449.

\end{thebibliography}


\providecommand{\bysame}{\leavevmode\hbox to3em{\hrulefill}\thinspace}
\providecommand{\MR}{\relax\ifhmode\unskip\space\fi MR }
\providecommand{\MRhref}[2]{%
  \href{http://www.ams.org/mathscinet-getitem?mr=#1}{#2}
}
\providecommand{\href}[2]{#2}

\end{document}